\documentclass[12pt]{amsart}

\usepackage{amsmath, amssymb, amscd, verbatim, xspace,amsthm}
\usepackage{latexsym, epsfig, color}

\newcommand{\C}{\ensuremath{{\mathbb{C}}}}
\newcommand{\Z}{\ensuremath{{\mathbb{Z}}}\xspace}
\renewcommand{\P}{\ensuremath{{\mathbb{P}}}}

\newcommand{\R}{\ensuremath{{\mathbb{R}}}}
\newcommand{\F}{\ensuremath{{\mathbb{F}}}}

\newcommand{\ra}{\rightarrow}

\newcommand\Det{\operatorname{Det}}

\newcommand\Hom{\operatorname{Hom}}

\newcommand\Aut{\operatorname{Aut}}

\newcommand\im{\operatorname{im}}

\newcommand\Sym{\operatorname{Sym}}
\newcommand\Sur{\operatorname{Sur}}

\newcommand\tensor{\otimes}
\newcommand\isom{\simeq}
\newcommand\sub{\subset}
\newcommand\tesnor{\otimes}

\newcommand\GL{\operatorname{GL}}

\newcommand\cok{\operatorname{cok}}

\newcommand\bq{\begin{equation}}
\newcommand\eq{\end{equation}}

\newtheorem{proposition}{Proposition}[section]
\newtheorem{theorem}[proposition]{Theorem}
\newtheorem{corollary}[proposition]{Corollary}

\newtheorem{lemma}[proposition]{Lemma}

\theoremstyle{remark}
\newtheorem{remark}[proposition]{Remark}

\usepackage{fullpage,pdfsync}

\newenvironment{definition}{\vspace{2 ex}{\noindent{\bf Definition. }}}{\vspace{2 ex}}

\newtheorem{nts}{Note to self}

\newcommand{\melanie}[1]{}

\newcommand{\E}{\mathbb{E}}
\newcommand{\cs}{\operatorname{colspace}}

\newcommand{\bS}{\bar{S}}
\newcommand{\bL}{\bar{L}}
\newcommand{\dc}{\gamma}

\title{The distribution of sandpile groups of random graphs}

\author{Melanie Matchett Wood}
\address{Department of Mathematics\\
University of Wisconsin-Madison \\ 480 Lincoln Drive \\
Madison, WI 53705 USA\\
and
American Institute of Mathematics\\360 Portage Ave \\
Palo Alto, CA 94306-2244 USA} 
\email{mmwood@math.wisc.edu}

\begin{document}
\begin{abstract}
We determine the distribution of the sandpile group (a.k.a. Jacobian) of the Erd\H{o}s--R\'{e}nyi random graph $G(n,q)$ as $n$ goes to infinity.  
Since any particular group appears with asymptotic probability $0$ (as we show), it is natural ask for the asymptotic distribution of Sylow $p$-subgroups of sandpile groups.  We prove the distributions of Sylow $p$-subgroups converge to specific distributions conjectured by Clancy, Leake, and Payne.
These distributions are related to, but different from, the Cohen-Lenstra distribution.
  Our proof involves first finding the expected number of surjections from the sandpile group to any finite abelian group (the ``moments'' of a random variable valued in finite abelian groups).    To achieve this, we show a universality result for the moments of cokernels of random symmetric integral matrices that is strong enough to handle dependence in the diagonal entries.
We then show these moments determine a unique distribution despite their $p^{k^2}$-size growth.

\end{abstract}

\maketitle

\section{Introduction}
Given a graph $\Gamma$, there is a naturally associated abelian group $S_\Gamma$, which has gone in the literature by many names, including the sandpile group, the Jacobian, the critical group, the Picard group. 
The order of $S_\Gamma$ is the number of spanning trees of $\Gamma$.
In \cite{Lorenzini2008}, Lorenzini asked about the distribution of sandpile groups of random graphs. 
  In this paper, we determine this distribution.  
In \cite{Clancy2013}, Clancy, Leake, and Payne noticed that sandpile groups did not appear to be distributed according to the well-known Cohen-Lenstra heuristics \cite{CL84}. They 
conjectured  certain new heuristics would govern how often various  abelian groups appear as sandpile groups.  In particular, we prove the distribution is as they conjectured.

For  $0<q<1$, we let $\Gamma\in G(n,q)$ be an Erd\H{o}s--R\'{e}nyi random graph on $n$ vertices with independent edge probabilities $q$.  One might first ask, for a finite abelian group $G$, what is
$
\lim_{n\ra \infty} \P(S_\Gamma\isom G)?
$
In fact, as we prove in this paper (Corollary~\ref{C:Zero}), 
$$
\lim_{n\ra \infty} \P(S_\Gamma\isom G)=0,
$$
showing this is too fine a question. 
 We will instead ask a coarser question about $S_\Gamma$. 
For example, the sandpile group  of $\Gamma$ is asymptotically almost never $\Z/8\Z \oplus \Z/ 3\Z \oplus \Z/11\Z$, but we could ask, how often is its $3$-part $\Z/3\Z$?
  A finite abelian group $G$ is simply the direct sum of its Sylow $p$-subgroups.  Let $S_{\Gamma,p}$ be the Sylow $p$-subgroup of $S_\Gamma$.  We will ask, for a fixed prime $p$, how often is $S_{\Gamma,p}$ a particular finite abelian $p$-group $G$.
\begin{theorem}\label{T:Main}
Let $p$ be a prime and $G$ a finite abelian $p$-group.
 Then for a random graph $\Gamma\in G(n,q)$, with $S_{\Gamma,p}$ the Sylow $p$-subgroup of its sandpile group,
$$\lim_{n\ra \infty} \P(S_{\Gamma,p}\isom G)=\frac{\#\{\textrm{symmetric, bilinear, perfect $\phi: G\times G
\ra \C^*$} \}}{|G||\Aut(G)|}\prod_{k\geq 0}(1-p^{-2k-1}).
$$
\end{theorem}
Note the product on the right does not involve $G$ and plays the role of a normalization constant.
Also, the entire right-hand side above does not depend on $q$, the edge-probability of the random graph.
If $G=\bigoplus_{i} \Z/p^{\lambda_i}\Z$ with $\lambda_1\geq\lambda_2\geq\cdots$ and  
 $\mu$ is the transpose of the partition $\lambda$, then we can express the fraction on the left very concretely as
\begin{equation}\label{E:main}
\frac{\#\{\textrm{symmetric, bilinear, perfect $G\times G
\ra \C^*$} \}}{|G||\Aut(G)|}=
p^{-\sum_i \frac{\mu_i(\mu_i+1)}{2}} {\prod_{i=1}^{\lambda_1} \prod_{j=1}^{\lfloor \frac{\mu_i-\mu_{i+1}}{2} \rfloor} (1-p^{-2j})^{-1}}
\end{equation}
(see proof of Lemma~\ref{L:Aut}).
For large $p$, every factor other than the leading power of $p$ is near $1$.
For example, if $G=\Z/p^r\Z$, the right hand side of Theorem~\ref{T:Main} is $\approx p^{-r}$.
If $G=(\Z/p\Z)^r$, the right hand side of Theorem~\ref{T:Main} is $\approx p^{-r(r+1)/2}$. This explains why cyclic groups are seen as sandpile groups so much more often than higher rank groups of the same size.
For example, the Sylow $7$-subgroup of a sandpile group is $\Z/49$ about $7$ times as often as often as it is $(\Z/7\Z)^2$.


In fact, we show quite a bit more than Theorem~\ref{T:Main}.  In particular (see Corollary~\ref{C:Main}) for any finite set of primes we give the asymptotic probabilities of particular Sylow subgroups at all of those primes, and we find that the Sylow subgroups at different primes behave (asymptotically) independently. 

To prove Theorem~\ref{T:Main}, we first prove a complete set of moments for the random groups $S_\Gamma$.
Let $\Sur(A,B)$ denote the set of surjective homomorphisms from $A$ to $B$.
\begin{theorem}\label{T:Mom}
Let $G=\bigoplus_{i=1}^r \Z/a_i\Z$ be a finite abelian group with $a_r | a_{r-1} | \cdots | a_1$.
Then for a random graph $\Gamma\in G(n,q)$, with $S_{\Gamma}$ its sandpile group,
$$
\lim_{n\ra \infty} \E(\#\Sur(S_{\Gamma}, G))=\prod_i a_i^{i-1}.
$$
\end{theorem}
The product $\prod_i a_i^{i-1}$ occurs as $|\wedge^2 G|$. 
 We refer to $\E(\#\Sur(S_{\Gamma}, G))$ as the \emph{$G$-moment} of $S_{\Gamma}$. 
 For comparison, if $H$ is a random group drawn according to the Cohen-Lenstra heuristics, then for all finite abelian groups $G$ the $G$-moment of $H$ is $1$ \cite[Proposition 4.1(ii) and Corollary 3.7(i)]{CL84} (see also \cite[Section 8]{EVW09}), whereas in our case the $G$-moments depend on the group $G$.
We also obtain an exponentially decreasing (in $n$) error term (see Theorem~\ref{T:MomGraphs})  for Theorem~\ref{T:Mom}. 

We then show (in Section~\ref{S:MomDet}) that the moments in Theorem~\ref{T:Mom} determine the distribution as given in Theorem~\ref{T:Main}, despite the moments' growing too fast to use the usual probabilistic methods to show that moments determine a unique distribution.
We also deduce many other statistics of sandpile groups of random graphs, including the distribution of their $p$-ranks  (see Corollaries~\ref{C:first} and \ref{C:prank}).  For example, the probability that 
$p$ divides $|S_\Gamma|$ goes to $1-\prod_{k\geq 0}(1-p^{-2k-1}).$
Even more concretely, the probability that a random graph $\Gamma\in G(n,q)$ has an even number of spanning trees goes to $\approx .5806$.
  We conclude in Corollary~\ref{C:cyclic} that the probability that $S_\Gamma$ is cyclic is asymptotically at most
$$
\zeta(3)^{-1}\zeta(5)^{-1}\zeta(7)^{-1}\zeta(9)^{-1}\zeta(11)^{-1}\cdots \approx .7935212,
$$
where $\zeta$ is the Riemann zeta function, differing from a conjectured value \cite[Conjecture 4.2]{Wagner2000}, and in Corollary~\ref{C:sqf} that the probability that the number of spanning trees of $\Gamma$ is square-free is asymptotically at most
$
\zeta(2)^{-1}\zeta(3)^{-1}\zeta(5)^{-1}\zeta(7)^{-1}\zeta(9)^{-1}\cdots \approx .48240306,
$
again differing from a conjectured value \cite[Conjecture 4.4]{Wagner2000}.
See also \cite[Section 4]{Lorenzini2008} for some questions and results on the topic of how often the sandpile group of a graph is cyclic.

\subsection{Sandpile groups}
For a general introduction to sandpile groups and some beautiful pictures of sandpiles, see the Notices' ``What is \dots a sandpile?'' \cite{Levine2010}.
There is also an overview given in \cite{Norine2011} of the way the group has arisen in various contexts.  One convenient definition is that $S_\Gamma$ is the cokernel of the reduced Laplacian $\Delta_\Gamma$ of $\Gamma$ (see \cite{Lorenzini1990} and also Section~\ref{S:sandpile}).   Thus $|S_\Gamma|=|\Det(\Delta_\Gamma)|$, which is the number of spanning trees of $\Gamma$ by  Kirchhoff's matrix tree theorem.   

The name ``sandpile'' comes from work studying the dynamics of a sandpile, which is a situation in which there is a number of chips at each vertex of a graph, and a vertex with at least as many chips as its degree can topple, giving a chip to each of its neighbors.  
(This is also called a chip-firing game, as originally studied in \cite{Bjorner1991}.)
The sandpile group parametrizes certain configurations of chips, called recurrent sandpiles, and is intimately related to the dynamics of the sandpile.  This sandpile model was first studied in statistical physics in 1988 \cite{Bak1988} (see also \cite{Dhar1990,Gabrielov1993,Gabrielov19930, Biggs1999}).
The sandpile group is also related to the Tutte polynomial of the graph. A generating function for counting elements of the sandpile group, as recurrent sandpiles, by their number of chips (on non-sink vertices) is given by $T(1,y)$, where $T$ is the Tutte polynomial
 \cite{Lopez1997,Gabrielov1993,Gabrielov19930}. See \cite{Holroyd} for a survey of some of these aspects of sandpiles.  
  
In an analogy between Riemann surfaces and graphs, the sandpile group has been studied and called the Jacobian (or Picard group or critical group) of the graph \cite{Bacher1997,Biggs1997}.  In this context, the group is a discrete analog of the Jacobian of a Riemann surface.  In fact, this analogy can be made precise, and the group of components of the N\'{e}ron model of a Jacobian of a curve over a local field is given as a Jacobian of a graph \cite{Lorenzini1989, Bosch2002}.  In this analogy, the order of the sandpile group appears in the ``analytic class number formula'' for graphs \cite{Horton2006}, and  versions of Riemann-Roch and the Riemann-Hurwitz formulas are known for the Jacobian of graphs \cite{Baker2007, Baker2009}.  

In part motivated by these many connections, the sandpile group has also been studied as an interesting invariant of graphs in its own right and determined for many families of graphs (see the Introduction to \cite{Alfaro2012} for pointers to some of this vast literature). 

\subsection{Why those probabilities: the relation to the Cohen-Lenstra heuristics}
Some experts had speculated that sandpile groups of random graphs might satisfy a Cohen-Lenstra heuristic.
The Cohen-Lenstra heuristics  \cite{CL84} were developed to predict the distribution of ideal class groups of quadratic number fields, which are finite abelian groups that measure the failure of unique factorization in quadratic rings of algebraic integers such as $\Z[\sqrt{-5}]$.  The basic principle is that a finite abelian group $G$ should occur with probability proportional to $|\Aut(G)|^{-1}$, barring any known bias in how groups appear. 
It is a well-known phenomenon that objects often occur inversely proportionally to their number of automorphisms.
 As in our case, with this heuristic, each group must appear with probability $0$ because the sum of $|\Aut(G)|^{-1}$ over all finite abelian groups is infinite, but as in this paper, the usual approach is to study the occurrence of a given Sylow $p$-subgroup $G$, which is expected to occur with \emph{positive} probability proportional to $|\Aut(G)|^{-1}$.  

As a first guess, this is a good one, and in fact the expected value given in Theorem~\ref{T:Mom} when $G$ is cyclic agrees with the average from the Cohen-Lenstra distribution, as was noticed empirically in \cite{Clancy2014}.   But higher averages do not agree with those from the Cohen-Lenstra distribution.  The Cohen-Lenstra distribution has been generalized to many other distributions where there is some additional feature of the relevant finite abelian group.   Even in the original Cohen-Lenstra paper \cite{CL84}, they modified the heuristic to predict the distribution of Sylow $p$-subgroups of class groups of real quadratic and higher degree abelian number fields for ``good'' primes $p$.
Gerth \cite{Ger87,Ger87b} gave different heuristics to predict the distribution for ``bad primes''.
Cohen and Martinet gave different heuristics that predict the class groups of any kind of extension of any number field \cite{CM90}.  
New heuristics have been suggested by Malle \cite{Mal08,Malle2010} and Garton \cite{Garton2012} to replace Cohen and Martinet's heuristics when there are roots of unity in the base field.
(Note that our moments in Theorem~\ref{T:Mom} agree with the ``$q=1$ case'' of the moments in \cite[Corollary 3.1.2]{Garton2012}, where the quotes are because the work in \cite{Garton2012} is motivated by work over a function field over $\F_q$.)
Most closely related to the situation for sandpile groups are Delaunay's heuristics for the distribution of Tate-Shafarevich groups of elliptic curves \cite{Delaunay2001}  (see also \cite{Bhargava2013}).  These groups are abelian and conjecture ally finite, and if finite have a non-degenerate, alternating, bilinear pairing.  So Delaunay formulated heuristics that replaced $\Aut(G)$ with automorphisms of $G$ that preserve the pairing.

In fact, the sandpile group comes with a canonical perfect, symmetric, bilinear pairing (see \cite{Lorenzini2000,Bosch2002,Shokrieh2010}), and so Clancy, Leake, and Payne  \cite{Clancy2013} guessed that this pairing should play a role in the distribution.  They
conjectured, based on their empirical results, that a particular group $G$ with pairing $\langle,\rangle$ should appear with probability proportional to
$|G|^{-1}|\Aut(G,\langle,\rangle)|^{-1}$.  Unlike the situation for alternating pairings, where each isomorphism class of group has a unique isomorphism type of pairing,  there are many isomorphism types of symmetric pairings, especially for $2$-groups.  The right-hand side of Theorem~\ref{T:Main} is what we obtain when summing 
the heuristic of \cite{Clancy2013} over all pairings for a given group.  It would be very interesting to have a refinement of Theorem~\ref{T:Main} that determines how often the various pairings occur for each group, and in particular to see if they indeed agree with the prediction of \cite{Clancy2013}.

\subsection{Connections to random matrices}
When $\Gamma$ is a random graph, the reduced Laplacian $\Delta_\Gamma$ is a random matrix, so one naturally arrives at the question of cokernels of random matrices with integer coefficients. 
In \cite{Clancy2014}, Clancy, Leake, Kaplan, Payne and the current author show that for a random symmetric matrix over the $p$-adic integers $\Z_p$, drawn with respect to Haar measure, the cokernels are distributed as in Theorem~\ref{T:Main}.  This is an analog of the work of Freidman and Washington \cite{FW89} that showed  that cokernels of random matrices over  $\Z_p$, drawn with respect to Haar measure, are distributed according to the Cohen-Lenstra heuristics.

For more general distributions of random matrices over $\Z_p$, Maples 
has a universality result showing that  random matrices over $\Z_p$ with independent, identically distributed entries have cokernels distributed according to the Cohen-Lenstra heuristics \cite{Map13a}.  
 Since the Sylow $p$-subgroup of the cokernel is trivial if and only if the matrix is invertible modulo $p$, determining the distribution of the Sylow $p$-subgroups is a refinement of the question of singularity of random matrices, which has been  well studied over $\R$, and was studied by Maples over finite fields in \cite{Maples2010}.  
Maples's work builds on and uses ideas from the work on estimating the singularity probability of a random matrix with $\pm 1$ i.i.d. entries of Kahn, Koml\'{o}s, Szemeredi \cite{Kahn1995} and Tao and Vu \cite{Tao2006,Tao2007} (see also \cite{Komlos1967,Komlos1968,Bourgain2010} for work on the singularity probability of a random matrix over $\R$).  This work all relies crucially on the independence of the entries of the matrix.

Our matrices $\Delta_\Gamma$ are symmetric, which adds significant difficulty over the case of independent entries.
In the case of singularity probability of symmetric matrices over $\R$ with independent entries  on and above the diagonal, Costello, Tao, and Vu obtained the first good bound \cite{Costello2006}, with improvements by Costello \cite{Costello2013} and Nguyen \cite{Nguyen2012}, and the current best bound due to Vershynin \cite{Vershynin2011}. 
The methods of these papers have the potential to address the question of $p$-ranks of symmetric matrices and their cokernels.  (For example, Maples has posted an announcement of results \cite{Map13b} giving the distribution of ranks of random symmetric matrices over $\Z/p\Z$.)  However, the previous methods for studying the singularity probability are not suited for determining more than the $p$-ranks of symmetric matrices, and so for example cannot distinguish between the groups $\Z/p\Z$ and $\Z/p^2\Z$.  (There is also the added complication in our case that the diagonal entries of 
the matrix $\Delta_\Gamma$ are \emph{not} independent from the rest of the entries.)

In order prove our main result, we therefore take a rather different approach than the work discussed above, though with some similar themes. 
On the path to our results on sandpile groups, we also give a universality result purely in the context of random matrices.  
In particular, we determine that the universal distribution of Sylow $p$-subgroups of cokernels of symmetric random matrices over $\Z$ (with independent entries on and above the diagonal) is the distribution in  Theorem~\ref{T:Main} (see Theorem~\ref{T:MomMat} and Remark~\ref{R:Mat}).   We also prove asymptotic independence of the joint distribution for finitely many primes $p$.

\subsection{Our method to determine the moments}

We prove Theorem~\ref{T:Mom} via a result in which a much more general random symmetric matrix replaces the graph Laplacian (see Theorem~\ref{T:MomMat}).  When then prove (in Theorem~\ref{T:Momdet}) that these moments in fact determine a unique distribution.   These are universality results that show for a large class of random symmetric matrices over $\Z$, their cokernels  have the same moments and the same distribution, asymptotically.
We are thus able to use the case of cokernels of \emph{uniform} random symmetric matrices over $\Z/a\Z$, whose distributions were determined in \cite{Clancy2014}, and by our universality results,  deduce that the moments and distribution of this  simple case hold in great generality.

There are some interesting features of our method, in particular in comparison to previous work.
    First of all, we only have to consider linear Littlewood-Offord problems, and not quadratic ones (as in \cite{Costello2006,Costello2013,Nguyen2012,Vershynin2011}), 
    even though our matrices are symmetric.  Second, our method can easily handle the dependence of 
    the diagonal of the Laplacian on the rest of the entries. 
 Third, we in fact obtain the moments, which are interesting averages in their own right and have 
 been studied at length for finite abelian group valued random variables in the work related to the 
 Cohen-Lenstra heuristics.  (For example, Davenport and Heilbronn \cite{DH71} determined the $\Z/3\Z$-moment of the class groups of quadratic fields.  See also \cite{Bhargava2005, EVW09,EVW12, Fouvry2006, FK07, Garton2012} for other examples of results in number theory about certain $G$-moments of class groups.) 
 Fourth, the moments only depend on the reduction of the matrix entries from $\Z$ to $\Z/a\Z$ for some $a$, so we are able to work with random symmetric matrices over $\Z/a\Z$.  (Of course, we need all the moments, so we must work over  $\Z/a\Z$ for each $a$.)

Theorem~\ref{T:Mom} gives the expected number of surjections $S_\Gamma \ra G$.
The sandpile group is $S_\Gamma=\Z^{n-1}/\Delta_\Gamma \Z^{n-1}$, so it suffices to determine
to probability that a surjection $F:\Z^{n-1}\ra G$ descends to $S_\Gamma$ (for each $F$).
Equivalently, we determine the probability that $F\Delta_\Gamma=0$.  This is a system of linear equations in the coefficients of $\Delta_\Gamma$.  The system is generated by on the order of $n$ equations and is in $\binom{n}{2}$ variables.  (This contrasts with the usual Littlewood-Offord problem
which is $1$ equation in $n$ variables.)  Unfortunately, the natural generators for this system have only order $n$ of the $\binom{n}{2}$ coefficients non-zero! 
 The system of equations is parametrized by $\Hom(\Z^{n-1},G^*)=(G^*)^{n-1}$, where $G^*$ is the group of characters on $G$. Further, some nontrivial $C\in(G^*)^{n-1}$ (we call these \emph{special}) turn out to  give equations in which \emph{all} of the coefficients are $0$, and which $C$ are special depends on the choice of $F$.
 
So while we have linear Littlewood-Offord type problems over $\Z/a\Z$ (with $a$ not necessarily prime, and with a growing number of linear equations instead of a single equation), the difficulty is to understand what structural properties of $F$ and $C$ influence how many of the coefficients of these equations are $0$. 
We develop two new concepts, \emph{depth} and \emph{robustness} to capture this key structure. 
We will give a brief overview of these concepts now; full details are included as the concepts arise in the paper.

 For $\sigma\sub [n-1]$, let $V_\sigma$ denote the column vectors in $\Z^{n-1}$ that have $\sigma$ entries $0$.  
\emph{Depth}  captures the structural properties of $F$ that influence how
 many non-zero coefficients appear in our system of equations.  For an integer $D$ with prime factorization $\prod_i p_i^{e_i}$, let $\ell(D)=\sum_i e_i$.
 
\begin{definition}
The \emph{depth} (depending on a parameter $\delta>0$) of a surjection $F:\Z^{n-1}\ra G$ is the maximal positive $D$ such that
there is a $\sigma\sub [n-1]$ with $|\sigma|< \ell(D)\delta (n-1)$ such that $D=[G:FV_\sigma]$, or is $1$ if there is no such $D$. 
\end{definition}

\emph{Robustness} captures the structural properties of $C$ that influence how
 many non-zero coefficients appear in a particular equation, given $F$.  Viewing $F\in\Hom(\Z^{n-1},G)$ and $C\in \Hom(\Z^{n-1},G^*)$, we can add them to obtain $F+C\in \Hom(\Z^{n-1},G\oplus G^*)$.
  
\begin{definition}
Given $F$, we say $C$ is \emph{robust} for $F$ (depending on a parameter $\gamma>0$),  if for every $\sigma\sub[n-1]$ with $|\sigma|<\gamma (n-1)$,
$$
\ker(F+C|_{V_\sigma} )\ne \ker(F|_{V_\sigma} ).
$$
\end{definition}

We identify the special $C$ exactly in terms of $F$. Despite their rarity, the special $C$ give the limit in Theorem~\ref{T:Mom} (the main term in Theorem~\ref{T:MomGraphs}).  The remaining cases form a complicated error term that we must bound.
  For $F$ of depth $1$, for non-special $C$ we prove the associated equation has at least order of $n$ non-zero coefficients, and for robust $C$ we prove the associated equation has at least order of $n^2$ non-zero coefficients. 
 For each larger depth, we compare $F$ to a combination of a depth $1$ ``$F$'' for a subgroup of $G$ (where we use the above) and an ``$F$'' for a quotient group of $G$ (where we use an Odlyzko-type bound).  There is a delicate balance between the number of non-zero coefficients we can get in each case and the number of pairs $(F,C)$ that fall into that case.

Finally, to deal with the dependence of the diagonal in $\Delta_\Gamma$, we actually do all of the above for a matrix with independent diagonal entries and then  enlarge $F$ to condition on what we require of the diagonal.

\subsection{Our method to determine the distribution from the moments}
The question of when the moments of a distribution determine a unique distribution is well-studied in probability and called the moment problem.
Roughly, if the sequence of moments of a random variable
does not grow too fast, then the distribution of the random variable is determined by the moments.  For example, Carleman's condition states that if
$
\sum_{k=1}^{\infty} m_{2k}^{-1/(2k)}
$
diverges, then there is a unique distribution on $\R$ having $m_k$ as the $k$th moment \cite[Section 2.3e]{Durrett2007}.  The standard counterexample is based on the lognormal density and has $k$th moment $e^{k^2/2}.$  In particular, there are  many $\R$-valued random variables $X$ with distinct distributions, such that for every $k$, we have $\E(X^k)=e^{k^2/2}$.

In our problem, the moments grow like the lognormal counterexample.  One can see this even if we were only interested in the $p$-ranks of sandpile groups.  Recall our moments are indexed by groups, but we will compare some of them to a usual moment.
Note that $\Hom(S_\Gamma,(\Z/p\Z)^k)=X^k$ for $X=p^{p\textrm{-rank}(S_\Gamma)}$.
 By adding Theorem~\ref{T:Mom} over all subgroups $G$
of $(\Z/p\Z)^k$ we conclude $\E(X^k)$ is of order $p^{(k^2-k)/2}$.
 However, the fact that our random variable $X$ can only take values in powers of $p$ makes the problem of recovering the distribution not completely hopeless.  

If we  were interested just in $p$-ranks (and did not want to distinguish between $\Z/p\Z$ and $\Z/p^2\Z$ for example), we could apply a method of Heath-Brown \cite[Lemma 17]{Heath-Brown1994-1}.  His strategy can be used to show that if $X$ is a random variable valued in $\{1,p,p^2,\dots\}$, and there is a constant $C$ such that for all integers $k\geq 0$ we have $\E(X^k) \leq C p^{k^2/2}$, then the distribution of $X$ is determined by its moments.  
Heath-Brown uses coefficients of precisely constructed analytic functions of one variable to lower-triangularize the infinite system of equations given by the moments.  
(See also \cite[Section 4.2]{Fouvry2006} which has a similar result but with a method that does not generalize to suit our needs.)

In order to recover the distribution of the entire Sylow $p$-subgroups of the sandpile group, we develop a generalization of Heath-Brown's method that replaces the analytic functions of one variable with analytic functions of several complex variables.  
However, the straightforward generalization which uses Health-Brown's functions for each variable is too weak for our purposes.  We perfectly optimize a function in each variable for our needs, and our moments are \emph{just} small enough for it to work.  In the end, we prove that mixed moments determine a unique joint distribution in cases where, as above, the moments are growing too fast to use Carleman's condition but where we have a restriction on the values taken by the random variables.
 
\subsection{Further questions}

This work raises many further questions. 
While we obtain an error bound in $n$ for Theorem~\ref{T:Mom} (see Theorem~\ref{T:MomGraphs}), we have not made explicit the dependence of the constant in that error bound on $G$.  It would be interesting to know if such an explicit dependence could translate into an error bound in $n$ for Theorem~\ref{T:Main}, and of what size.

We also work with $p$ fixed, and therefore our methods are not ideal for questions that require consideration of $p$ large compared to $n$, such determining the probability that $S_\Gamma$ is cyclic (for which we obtain only an upper bound, though at what, in light of our results, seems very likely to be the correct answer).  It would be very interesting to know if our approach could be combined with ideas from \cite{Map13a}, which are uniform in $p$, to determine the probability that $S_\Gamma$ is cyclic.
As a byproduct of understanding the group structure, we have determined the distribution of the size of $|S_\Gamma|$ in the $p$-adic metric, but it is also natural to ask about the distribution of $|S_\Gamma|$ as a real number.  While from the above we see that it is any particular size with asymptotic probability $0$, we can ask about the probability that it lies in appropriately sized intervals.  
As discussed above, it would also be nice to have results on the distribution of the pairing on $S_\Gamma$.

Another interesting question is whether results such as Theorems~\ref{T:Main} and \ref{T:Mom} hold
for other models of random graphs or whether the values of the probabilities and the moments change (see \cite[Remark 2]{Clancy2014}).
Our results already allow the edge probabilities to vary as long as they are independent and bounded above and below by a constant.  However, it would be interesting to know if one obtains the same distribution on sandpile groups for sparser graphs, in particular in the case of $G(n,q)$ when $q\geq (1+\epsilon)\log (n)/n$ in which the graph is still asymptotically almost surely connected.   
There is a analogous question for denser graphs (and if $q$ gets too large, the graphs will be too close to complete graphs and will definitely not follow the distribution of Theorem~\ref{T:Main}).
 It would also be interesting to determine the distribution of sandpile groups of $r$-regular graphs.  

In this paper we work with random symmetric matrices as the basic object, and we have already had to deal with one kind of dependency (beyond the symmetry) in our matrices---the dependency of the diagonal in the graph Laplacian on the other entries.  We specifically developed our method to handle this dependency easily, and it should be able to handle other linear dependencies on the columns of a symmetric matrix as well, as long as they apply to all the columns.  It would be nice to understand whether our approach can be extended handle to linear dependencies that only apply to some of the columns, and in general to what extent dependencies affect the outcome of the distribution of the cokernels of random symmetric matrices.  
Another interesting case to consider is one in which some of the entries of the matrix are fixed, such as for the adjacency matrix of a random graph in which case the diagonal entries are $0$.  The cokernel of the adjacency matrix is called the \emph{Smith group}, and has been studied e.g. in \cite{Chandler2014,Ducey2013}.

In this paper, our method finds the actual values of the probabilities occurring in Theorem~\ref{T:Main} by 
using our universality results that say the values are the same for a large class for random matrices, and then citing a computation for the case of uniform random symmetric matrices over $\Z/a\Z$.  There are further statistics of these uniform random matrices over $\Z/a\Z$, which if determined, would, using our Corollary~\ref{C:first}, immediately give more statistics of sandpile groups of random graphs. 
See the end of Section~\ref{S:Haar} for details.

\subsection{Outline of the paper}
In Sections~\ref{S:MomStart} through \ref{S:FinalMom} we prove Theorem~\ref{T:Mom}
(and the analog of Theorem~\ref{T:Mom} for cokernels of symmetric random matrices).  In Section~\ref{S:MomDet}, we prove that the moments of Theorem~\ref{T:Mom} in fact determine the relevant distributions.  In Section~\ref{S:Haar}, we show what those distributions are, by comparing to the case of cokernels of uniform random symmetric matrices over $\Z/a\Z$, for which the distribution and the moments have already been computed in \cite{Clancy2014}.  In particular, we deduce Theorem~\ref{T:Main} from Corollary~\ref{C:Main}.

\section{Background}\label{S:Back}
\subsection{Cokernels of matrices}
For an $n\times n$ matrix $M$ with entries in a ring $R$, 
let $\cs(M)$ denote the column space of $M$ (i.e. the image of the map $M: R^n \ra R^n$).
We define the \emph{cokernel} of $M$,
$$
\cok(M):= R^n/\cs(M).
$$

\subsection{Sandpile group}\label{S:sandpile}
Let $[n]$ denote the set $\{1,\dots,n\}$.  Let $\Gamma$ be a graph on $n$ vertices labeled by $[n]$.  The Laplacian $L_\Gamma$ is an $n\times n$ matrix with $(i,j)$ entry 
$$
\begin{cases}
1 \textrm{ if $\{ i,j \} $ is an edge of $\Gamma$}\\
0 \textrm{ if $i\ne j$ and $\{i,j\}$ is not an edge of $\Gamma$}\\
-\deg(i) \textrm{ if $i=j$}.
\end{cases}
$$
We have that $L_\Gamma$ is a matrix with coefficients in $\Z$.
Let $Z\sub \Z^n$ be the vectors whose coordinates sum to $0$.  Clearly, 
$\cs(L_\Gamma) \sub Z$.  We define the sandpile group $S_\Gamma:= Z/\cs(L_\Gamma)$.
This is clearly a finitely generated abelian group, and it is finite if and only if $\Gamma$ is connected.

\subsection{Random graphs}
We write $\Gamma\in G(n,q)$ to denote that $\Gamma$ is an Erd\H{o}s--R\'{e}nyi random graph on $n$ labeled vertices with each edge independent and occurring with probability $q$.

\subsection{Finite abelian groups}
For a prime $p$, a finite abelian $p$-group is isomorphic to 
$\bigoplus_{i=1}^r \Z/p^{\lambda_1}\Z$ for some positive integers $\lambda_1\geq \lambda_2\geq\dots\geq \lambda_r$.
We call the partition $\lambda$ the \emph{type} of the abelian $p$-group.  For a partition $\lambda$, we use $G_\lambda$ to denote a $p$-group of type $\lambda$ when $p$ is understood.


For an $a\in \Z$ and a finite abelian group $G$, we can form the {\bf tensor} product $G\tesnor_\Z \Z/a\Z.$  This is a tensor product of the two objects as $\Z$-modules, but is particularly simple to say what it does to a particular group.  We have
$$
\left(\bigoplus_i \Z/a_i\Z \right)\tesnor_\Z \Z/a\Z=\bigoplus_i \Z/(a_i,a)\Z,   
$$
where $(a_i,a)$ is the greatest common divisor of $a_i$ and $a$.  So for primes $p\nmid a$, the Sylow $p$-subgroups are killed, 
and if $p^{e}$ is the highest power of a prime $p$ dividing $a$, then summands $\Z/p^i\Z$ of $G$ for $i\leq e$ are untouched and summands $\Z/p^i\Z$ for $i>e$ are changed to $\Z/p^e\Z.$  In terms of the partition diagram for the type $\lambda$ of the Sylow $p$-subgroup of $G$, it is truncated so that all rows are length at most $e$.

The {\bf exterior power} $\wedge^2 G$ is defined to be the quotient of $G\tesnor G$ by the subgroup generated by elements of the form $g\tensor g$.
If $G_p$ are the Sylow $p$-subgroups of $G$, then  $\wedge^2 G=\bigoplus_p \wedge^2 G_p.$
If $G_p$ is type $\lambda$, generated by $e_i$ with relations $p^{\lambda_i}e_i=0$, then
$\wedge^2 G_p $ is generated by the $e_i \wedge e_j$ for $i<j$ with relations
$p^{\lambda_j} e_i \wedge e_j=0$.
So
$$
\wedge^2 G_p \isom \bigoplus_{i} (\Z/p^{\lambda_i} \Z)^{\oplus (i-1)}.
$$

For a partition $\lambda$, let $\lambda'$ be the {\bf transpose partition}, so $\lambda'_j$ is the number of $\lambda_i$ that are at least $j$.  
Note that $\sum_i (i-1)\lambda_i$ is the sum over boxes in the partition diagram of $\lambda$ of $i-1$, where $i$ in the row the box appears in.  Summing by column, we obtain
$\sum_i (i-1)\lambda_i=
\sum_j \frac{\lambda'_j(\lambda'_j-1)}{2}.
$ 
Of particular importance to us will be the size
$$
|\wedge^2 G_p | =p^{\sum_i (i-1)\lambda_i}= p^{\sum_j \frac{\lambda'_j(\lambda'_j-1)}{2}}.
$$
 
 The exponent of a finite abelian group is the smallest positive integer $a$ such that $aG=0$.
When $R=\Z/a\Z$,  any finite abelian groups $H,G$ of exponent dividing $a$ are also {\bf $R$-modules}, and their group homomorphisms are the same as their $R$-module homomorphisms.
When the ring $R$ is understood, we write $G^*$ for $\Hom(G,R)$.  
If the exponent of $G$ divides $a$, then $G^*$ is non-canonically isomorphic to $G$.

We use $\langle g_1,\dots, g_m \rangle$ to denote the subgroup generated by $g_1,\dots, g_m$. 

\subsection{Pairings}
A map $\phi: G \times G \ra \C^*$ is symmetric if $\phi(g,h)=\phi(h,g)$ for all $g,h\in G$.
When $G$ is an abelian group, the map $\phi$ is bilinear if for all $g_1,g_2,h\in G$ we have
$\phi(g_1+g_2,h)=\phi(g_1,h)\phi(g_2,h)$, and similarly for the right factor. 
The map $\phi$ is perfect if the only $g\in G$ with $\phi(g,G)=1$ is $g=0$, and similarly for the other factor.

\subsection{Notation}
We denote the order of groups and sets using either absolute value signs $|\cdot|$ or $\#$.
(This inconsistency is because sometimes the absolute value signs are confusing when coupled with the notation $\mid$ for ``divides'' or parentheses, and the sharps take up too much space in some formulas.)
  We use $\isom$ to denote ``is isomorphic to.''  We use $\P$ to denote probability and $\E$ to denote expected value.
The letter $p$ will always denote a prime.

\section{Obtaining the moments I:
Determining the structural properties of the equations }\label{S:MomStart}

In the next four sections, we will prove Theorem~\ref{T:Mom}.  Let $G$ be a finite abelian group and $\Gamma\in G(n,q)$.  
We will write $S$ for $S_\Gamma$ and $L$ for the Laplacian $L_\Gamma$.
The group $S$ is defined as a quotient of $Z$, so any surjection $S\ra G$ lifts to a surjection $Z\ra G$, so we have
$$
\E(\#\Sur(S,G))=\sum_{F\in \Sur(Z,G)} \P(\cs(L)\sub \ker(F)).
$$
Our approach will be to estimate the probabilities on the right, but we will start with a slightly more general set up.

Let $a$ be a positive integer and let $G$ be a finite abelian group of exponent dividing $a$. Let $R$ be the ring $\Z/a\Z$. 
We will retain this notation through Section~\ref{S:FinalMom}.
Note that $\Sur(S,G)=\Sur(S\tesnor \Z/a\Z,G)$.  Said another way, whether $\cs(L)\sub \ker(F)$ only depends on the entries of $L$ modulo $a$. 

 In this and the next three sections, we shall do all our ``linear algebra'' over $R$.  Since $R$ is not a domain, this necessitates working more abstractly instead of just with matrices.
   A particular source of difficulty compared to the case of linear algebra over a field is that not all exact sequences of $R$-modules split, i.e. there are subgroups of our finite abelian groups that are not direct summands.  We will work carefully to find summands when we need them.


For an $R$-module $A$, let $A^*:=\Hom(A,R)$.  
We define the $R$-module $V=R^n$, the elements of which we write as column vectors.  We have a distinguished basis $v_1,\dots,v_n$ of $V$, and a dual basis $v_1^*,\dots, v_n^*$ of $V^*$.
Also let $W=R^n$, the elements of which we write as column vectors as well.  We have a  basis $w_1,\dots,w_n$ of $W$, and a dual basis $w_1^*,\dots, w_n^*$ of $W^*$.
An $n\times n$ matrix $M$ over $R$ is a homomorphism from $W$ to $V$, i.e. $M\in \Hom(W,V)$. 

Let $F\in \Hom(V,G)$.  Then $\cs(M)\sub \ker(F)$ if and only if the composition $FM\in \Hom(W,G)$ is $0$.
Let $\zeta$ be a primitive $a$th root of unity. 
So, if $X\in \Hom(W,V)$ is a random matrix, the Fourier transform gives
$$
\P(FX=0) =\frac{1}{|G|^n} \sum_{C\in \Hom(\Hom(W,G),R)) } \E (\zeta^{C(FX)}).
$$
These $C$ give the equations a matrix has to satisfy in order for a surjection given by $F$ to extend to the cokernel of the matrix.  (Hence, ``equations'' in the title of this section.)

Since $W\isom R^n$, we have that the natural map $\Hom(W,R)\tensor G \ra  \Hom(W,G)$ is an isomorphism.  So, the natural map $\Hom(\Hom(W,G),R))\ra \Hom(\Hom(W,R)\tensor G,R))$ is an isomorphism.  Composing with the isomorphism 
$\Hom(W^*\tensor G,R)) \isom \Hom(W^*,\Hom(G,R)))$, we have an isomorphism 
$
\Hom(\Hom(W,G),R)) \ra \Hom(W^*,G^*).$  Via this isomorphism, we will view $C\in \Hom(W^*,G^*)$.
So for $w^*\in W^*$, we have $C(w^*)\in G^*$.
We write $e: G^* \times G\ra R$ for the map that evaluates a homomorphism.  

Because of our interest in random matrices whose entries \emph{with respect to a specific choice of basis of $V$} are independent, we must necessarily sometimes compute things with respect to this basis.  
In other parts of the proof, we will work with a different choice of basis more closely aligned with $G$ (through $F$).
For some parts of our proof, in particular because we are working over the non-domain $R=\Z/a\Z,$ it will be much simpler to work in a basis-free way.  

In particular, our interest is in symmetric matrices $X$.  For this even to make sense, we now identify $W=V^*$ and so $v_i=w_i^*$ and $v_i^*=w_i$.
Our matrix $X$ will be symmetric and so we have
%
\begin{align*}
&C(FX)=\sum_{i=1}^n \sum_{j=1}^n 
e(C(v_j) ,F(v_i) ) X_{ij}\\&=\sum_{i=1}^n \sum_{j=i+1}^n 
(e(C(v_j) ,F(v_i) ) +e(C(v_i) ,F(v_j) )  )X_{ij} +\sum_{i=1}^{n} 
e(C(v_i) ,F(v_i) ) X_{ii}.
\end{align*}
We will study these coefficients in detail. For $i<j$ we define, 
$E(C,F,i,j):=e(C(v_j) ,F(v_i) ) +e(C(v_i) ,F(v_j) )$, and we also define
$E(C,F,i,i):=e(C(v_i) ,F(v_i) )$.  Roughly, our goal is to see
that as many as possible of these coefficients are non-zero,  as often as possible.  To do this 
we will have to identify structural properties of $F$ and of $C$ that influence the number of non-zero coefficients.  There are on the order of $n^2$ coefficients, and so ideally we would like 
on the order of $n^2$ of them to be non-zero.  Unfortunately, given $F$, this is not the case for every $C$.  Given a ``good'' $F$, for most $C$ we will be able to show that on the order of $n^2$ of the coefficients are non-zero, but for some only on the order of $n$ of the coefficients are non-zero, and for some $C$ all of the coefficients are $0$.  The rest of this section is devoted to explaining the structural properties of $C$ that will determine which of the three cases above it falls into.
(This is all for ``good'' $F$.  In this section we will determine the structural property that makes $F$ good, and in Section~\ref{S:depth} we will come to the rest of the $F$, which we will have to stratify by further structural properties.)

We will now write these coefficients $E(C,F,i,j)$ more equivariantly via a pairing.
We have a map $\phi_{F,C} \in \Hom(V, G \oplus G^*)$ given by adding $F$ and $C$.
We also have a map $\phi_{C,F} \in \Hom(V, G^* \oplus G)$ given by adding $C$ and $F$.
There is a map
\begin{align*}
(G \oplus G^*) \times (G^* \oplus G) &\stackrel{t}{\ra} R\\
((g_1,\phi_1),(\phi_2,g_2)) &\mapsto \phi_2(g_1)+\phi_1(g_2).
\end{align*}
Note that for all $u,v\in V$,
$$
t(\phi_{C,F}(u), \phi_{F,C}(v))= e(C(u) ,F(v) ) +e(C(v) ,F(u) ) .
$$

Note has $V$ has distinguished submodules $V_\sigma$ generated by the $v_i$ with $i\not \in \sigma$ for each $\sigma\sub [n]$.  So $V_\sigma$ comes from not using the coordinates in $\sigma$.
Clearly, for any submodule $U$ of $V$, 
$$\ker(\phi_{F,C}|_{U} )\sub \ker(F|_{U} ).$$
Now we will define the key structural property of $C$ (with respect to $F$) that determines if enough of the coefficients $E(C,F,i,j)$ are non-zero.

\begin{definition}
Let $0<\dc<1$ be a real number which we will specify later in the proof.
Given $F$, we say $C$ is \emph{robust} (for $F$)  if for every $\sigma\sub[n]$ with $|\sigma|<\dc n$,
$$
\ker(\phi_{F,C}|_{V_\sigma} )\ne \ker(F|_{V_\sigma} ).
$$
Otherwise, we say $C$ is \emph{weak} for $F$. 
\end{definition}

 We will estimate the number of weak $C$.  

\begin{lemma}[Estimate for number of weak $C$]\label{L:estweakC}
Given $G$, there is a constant $C_G$ such that for all $n$ the following holds.
Given $F\in\Hom(V,G)$, the number of $C\in \Hom(V,G^*)$ such that $C$ is weak for $F$ is at most
$$
C_G\binom{n}{\lceil \dc n \rceil -1} |G|^{\dc  n }
$$
\end{lemma}
\begin{proof}
If $C$ is weak, then there exists some $\sigma\sub [n]$ with $|\sigma| =\lceil \dc n \rceil -1$ such that 
$$
\ker(\phi_{F,C}|_{V_\sigma} )= \ker(F|_{V_\sigma} ).
$$
We note in particular this implies that for $s\in V_\sigma$, we have that $Cs$ is determined by $Fs$.
(If $Fs=Fs'$ but $Cs\ne Cs'$, then $s-s'\in \ker F$ but $s-s'\not\in \ker\phi_{F,C}|_{V_\sigma}$.)
Let $H:=\im F|_{V_\sigma}$.  Further, there is a homomorphism $\psi: H \ra G^*$ so that
$Cs=\psi(Fs)$ for all $s\in V_\sigma$.  There are $\binom{n}{\lceil \dc n \rceil -1} $ choices for $\sigma,$
then $|G|^{\dc  n}$ choices for $Cv_i$ for $i\in\sigma$,
then $\#\Hom(H,G^*)$ choices for $\psi$, and then $C$ is determined.  Note that since $H$ is a subgroup of $G$ we can find $C_G$ such that $\#\Hom(H,G^*)\leq C_G 
$.
\end{proof}

Now we will find a sufficient condition for $C$ to be weak in terms of our pairing $t$.

\begin{lemma}\label{L:pairtoinject}
Let $F\in\Hom(V,G)$ and  $C\in\Hom(V,G^*)$. Let $U$ be a submodule of $V$ such that $FU=G$.  Then if $U'$ is a submodule of $V$ such that
$t$ is $0$ on $\phi_{C,F}(U) \times \phi_{F,C} (U')$, then the projection map $G\oplus G^* \ra G$, when restricted to 
$\phi_{F,C}(U')$, is an injection.
In particular $\ker(\phi_{F,C}|_{U'} )=\ker(F|_{U'} ).$
\end{lemma}

\begin{proof}
Suppose for the sake of contradiction that there is a $k\in U'$ with $Fk=0$ and $Ck=\psi \ne 0\in G^*$.  Since $\psi\ne0$, there must be some $g\in G$ such that $\psi(g)\ne 0$.  Since $FU=G$, there must be some $r\in U$ such that $Fr=g$.  Suppose $Cr=\psi'$.  Then $t(\phi_{C,F} (r), \phi_{F,C} (k))=\psi'(0)+\psi(g)\ne 0$.  So, we conclude $\phi_{F,C} (U')$ injects into $G$.

If $\ker(\phi_{F,C}|_{U'} )\ne \ker(F|_{U'} ),$
then  there is some $(0,\phi)\in \phi_{F,C}(U')$ with $\phi \ne 0$, which is a contradiction.
\end{proof}

\begin{corollary}\label{C:robustones}
Let $F\in\Hom(V,G)$ and  $C\in\Hom(V,G^*)$.
Let $U$ be a submodule of $V$ such that $FU=G$.  Then if
$$
\#\{i\in [n] \ |\  t(\phi_{C,F}(U), \phi_{F,C}(v_i))\ne 0 \}<\dc n
$$
then $C$ is weak for $F$.
\end{corollary}

\begin{proof}
Let $\sigma:=\{i\in [n] \ |\  t(\phi_{C,F}(U), \phi_{F,C}(v_i))\ne 0 \}$.  Then $t(\phi_{C,F}(U), \phi_{F,C}(V_\sigma))=0$, and so by Lemma~\ref{L:pairtoinject}
we have that $C$ is weak for $F$.
\end{proof}


Now we will identify the influential structural property of $F$ (which will make it ``good'' as discussed above), which is a (transpose and) generalization of the notion
of a linear code from vector spaces to $R$-modules.

\begin{definition}
We say that $F\in \Hom(V,G)$ is a \emph{code} of distance $w$, if for every $\sigma\sub [n]$ with $|\sigma|<w$, we have $FV_\sigma=G$.
In other words, $F$ is not only surjective, but would still be surjective if we throw out (any) fewer than $w$ of the standard basis vectors from $V$.  (If $a$ is prime so that $R$ is a field, then this is equivalent to whether the transpose map $F: G^* \ra V^*$ is injective and has image $\im(F)\sub V^*$ a linear code of distance $w$, in the usual sense.)
\end{definition}

We have the following lemma about codes which we will next combine with the property of robustness to get a good bound on the number of $E(C,F,i,j)$ that are non-zero.

\begin{lemma}\label{L:FindAB}
Let $H$ be a finite $R$-module with Sylow $p$-subgroup of type $\lambda$.  
Suppose  $F\in \Hom(V,H)$ is a \emph{code} of distance $\delta n$, and let $C\in \Hom(V,H^*)$.  
Let $r=\lambda_1'$.
Then we can find $A_1,\dots,A_r\in H$ and $B_1,\dots,B_r\in H^*$ such that for every $1\leq i \leq r$
$$
\#\{ j\in [n] \ |\ Fv_j=A_i \textrm{ and } Cv_j=B_i \}\geq  \delta n/|H|^2,
$$
and after the projection to the Sylow $p$-subgroup of $H$, the elements 
 $A_1,\dots,A_r$ generate the Sylow $p$-subgroup of $H$.

\end{lemma}

\begin{proof}
We find the $A_i$ and $B_i$ by induction, so that (after the projection to the Sylow $p$-subgroup of $H$) the elements $A_1,\dots,A_k$ generate a $p$-subgroup of type $\lambda_1,\dots,\lambda_k$ that is a summand of the Sylow $p$-subgroup of $H$.
Suppose we are done for $i\leq k$.  First, we count for how many $i$ is $Fv_i$ order $p^{\lambda_{k+1}}$ in the projection of $H/\langle A_1,\dots, A_k\rangle$ to the Sylow $p$-subgroup of $H$.  Suppose, for the sake of contradiction, that there were fewer than $\delta n$.  Then we have a $\sigma\sub [n]$ with $|\sigma|<\delta n$ such that $FV_\sigma\ne H$, contradicting the fact that $F$ is a code.  So, we have at least $\delta n$ values of $i$ such that  $Fv_i$  is order $p^{\lambda_{k+1}}$ in the projection of $H/\langle A_1,\dots, A_k\rangle$ to the Sylow $p$-subgroup of $H$.  There are at most $|H|^2$ possible values for $(Fv_i,Cv_i)$, so we let $(A_{k+1},B_{k+1})$ be the most commonly occurring value for the at least $\delta n$ values of $i$ we have found above.  
Any element of order $p^{\lambda_{k+1}}$ in an abelian $p$-group of exponent $p^{\lambda_{k+1}}$
generates a summand.
Since after projection to the Sylow $p$-subgroup of $H$, we have that $\langle A_1,\dots, A_k\rangle$ is a summand of the Sylow $p$-subgroup of $H$, 
and $A_{k+1}$ generates a summand of the quotient $H/\langle A_1,\dots, A_k\rangle$ after projection to the Sylow $p$-subgroup of $H$,
 we see that $\langle A_1,\dots, A_k, A_{k+1}\rangle$ is as desired.
\end{proof}

Now we will see that robustness does in fact determine that many of our coefficients of interest are non-zero.

\begin{lemma}[Quadratically many non-zero coefficients for robust $C$]\label{L:Quad}
Let $P$ be the set of primes dividing the order of $G$. 
If $F\in\Hom(V,G)$ is a  code of distance $\delta n,$ and if   $C\in\Hom(V,G^*)$ is robust for $F$, then
there are at least $\dc \delta n^2/(2|G|^2|P|)$ pairs $(i,j)$ with $i\leq j$ such that
$$
E(C,F,i,j)\ne 0.
$$
\end{lemma}

\begin{proof}

For each $p\in P$, let $G_p$ be the Sylow $p$-subgroup of $G$. 
  Now using $p\in P$ and $F\in \Hom(V,G)$ and $C\in \Hom(V,G^*)$, 
 we pick $A_i(p)$ and $B_i(p)$ as in Lemma~\ref{L:FindAB}, and let 
$$\tau_i(p):=\{ j\in [n] \ |\ {F}v_j=A_i \textrm{ and } {C}v_j=B_i \}.$$
Let $\tau(p):=\cup_{i} \tau_i(p)$.  Let $V_p$ be the submodule of $V$ generated by the $v_j$ for $j\in \tau(p)$.
In particular, note that ${F}V_p$, in the projection to $G_p$, is all of $G_p$. 

Now, let $W$ be the submodule of $V$ generated by the $V_p$ for all $p\in P$.
In particular, note that $FW=G$.
So if $C$ is robust for $F$,  by Corollary~\ref{C:robustones},
$$
\#\{i\in [n] \ |\  t(\phi_{C,F}(W), \phi_{F,C}(v_i))\ne 0 \}\geq\dc n.
$$
We have
$$
\sum_{p\in P} \#\{i\in [n] \ |\  t(\phi_{C,F}(V_p), \phi_{F,C}(v_i))\ne 0 \}
\geq \#\{i\in [n] \ |\  t(\phi_{C,F}(W), \phi_{F,C}(v_i))\ne 0 \}
$$
because if $v_i$ pairs non-trivially with $W$, it must pair non-trivial with one of the submodules generating $W$.  So for some $p\in P$, we have 
$$
\#\{i\in [n] \ |\  t(\phi_{C,F}(V_p), \phi_{F,C}(v_i))\ne 0 \}\geq\dc n/|P|.
$$

Then for that particular $p$, 
$$
\#\{i\in [n] \ |\  t(\phi_{C,F}(v_j), \phi_{F,C}(v_i))\ne 0 \textrm{ for some $j\in \tau(p)$} \}\geq\dc n/|P|.
$$
However, there are at least $\delta n/|G|^2$ values of $j'\in\tau(p)$ with $\phi_{C,F}(v_{j'})=\phi_{C,F}(v_j)$.
Since for $i<j$ we have $t(\phi_{C,F}(v_j), \phi_{F,C}(v_i))=E(C,F,i,j)$, and also
$t(\phi_{C,F}(v_i), \phi_{F,C}(v_i))=2E(C,F,i,i)$,
we conclude that there are at least $\dc \delta n^2/(2|G|^2|P|)$ pairs $(i,j)$ with $i\leq j$ such that
$
E(C,F,i,j)\ne 0.
$
\end{proof}


Next, we will study how many coefficients can be non-zero for weak $C$.  Of course for $C=0$, all the $E(C,F,i,j)$ are $0$.  However, given $F$, there are other $C$ for which this can happen, and next we will identify those $C$.

We now take a second equivariant point of view on $E(C,F,i,j)$.
There is a natural map
coming from the evaluation map $G \tensor G^* \ra R$, 
\begin{align*}
\Hom(V\tesnor G) \tensor \Hom(V,G^*) &\ra V^* \tensor V^*. 
\end{align*}
We can further compose with the quotient $V^* \tensor V^* \ra \Sym^2 V^*$ to obtain
\begin{align*}
\Hom(V\tesnor G) \tensor \Hom(V,G^*) &\ra V^* \tensor V^* \ra \Sym^2 V^*. 
\end{align*}
So given an $F\in\Hom(V,G)$, we have a map
\begin{align}\label{E:mexplicit}
m_F :\Hom(V,G^*) &\ra \Sym^2 V^*\notag\\
C &\mapsto  \sum_{i=1}^n \sum_{j=i+1}^n (e(C(v_j) ,F(v_i) ) +e(C(v_i) ,F(v_j) )  )v_i^*v_j^* +\sum_{i=1}^{n} 
e(C(v_i) ,F(v_i) ) (v_i^*)^2.
\end{align} 

First, we determine some elements $C\in \Hom(V,G^*)$ that are in the kernel of $m_F$, i.e. all the $E(C,F,i,j)$ are $0$.
For $F\in\Hom(V,G)$ and $u\in G^*$, we con compose $F$ with $u$ to obtain $u(F)\in\Hom(V,R)$.
We can then multiply by $v\in G^*$ to obtain $u(F)v\in \Hom(V,G^*)$.  So we have a map
\begin{align*}
s_F: \wedge^2 G^* &\ra \Hom(V,G^*)\\
u\wedge v &\mapsto u(F)v-v(F)u.
\end{align*}
We can check that $\im(s_F)\sub \ker (m_F)$ by choosing a generating set for $G^*$.
Let $G\isom \oplus_{i=1}^{m} \Z/a_i\Z$, with $a_m\mid a_{m-1} \mid \cdots \mid a_1$.
Let $G^*$ be given by generators $e_i^*$ and relations $\frac{a}{a_i}e_i^*=0$.
So using Equation~\eqref{E:mexplicit}, we will check that  $\im(s_F)\sub \ker (m_F)$.
Let $C=s_F(e_i^* \wedge e_j^*)$.
Then the $v_a^*v_b^*$ coefficient of $m_F(C)$ is
\begin{align*}
&e(C(v_b) ,F(v_a) ) +e(C(v_a) ,F(v_b) ) \\
&=
e(e_i^*(Fv_b) e_j^* - e_j^*(Fv_b) e_i^* ,F(v_a) ) +e(e_i^*(Fv_a) e_j^* - e_j^*(Fv_a) e_i^* ,F(v_b) )
\\
&=
e_i^*(Fv_b) e_j^*(Fv_a) - e_j^*(Fv_b) e_i^*(Fv_a)+e_i^*(Fv_a) e_j^*(Fv_b ) - e_j^*(Fv_a) e_i^* (Fv_b )
\\
&=0.
\end{align*}
Similarly, the coefficient of $(v_a^*)^2$ in $m_F(s_F(e_i^* \wedge e_j^*))$
is $0$.
So we conclude $\im(s_F)\sub \ker (m_F)$.
We call the $C$ in  $\im(s_F)$ \emph{special} for $F$.

Now we see how many special $C$ there are.
\begin{lemma}\label{L:sfinj}
If $FV=G$, then we have that $s_F$ is injective.  In particular, $\#\wedge^2 G| \#\ker(m_F)$.
\end{lemma}

\begin{proof}
It suffices to show that $\#\wedge^2 G| \#\im(s_F)$. 
Since everything in sight can be written as a direct sum of Sylow $p$-subgroups, we can reduce to the case that $G$ is a $p$-group of type $\lambda$ (and accordingly assume $R=Z/p^e\Z$). Let $r=\lambda_1'$.

By Lemma~\ref{L:FindAB}, we can find $\tau\sub [n]$ with $|\tau|=r$ such that $Fv_i$ generate $G$ for $i\in\tau$.  Let $W$ be the submodule of $V$ generated by the $v_i$ for $v\in \tau$.
Let $e_i$ generate $G$ with relations $p^{e-\lambda_i} e_i=0$.
Let $w_j\in W$ be such that $Fw_j=e_j$.
Let $W'\sub W$ be the subgroup of $W$ generated by the $w_j$.
Note that we have the maps
$$
W'/pW' \ra W/pW \stackrel{F}{\ra} G/pG.
$$
Since $W'/pW'$, $W/pW$, and  $G/pG$ are vector spaces over $\F_p$, with rank at most $r$, exactly $r$, and $r$ respectively, and
the composite map above is surjective, we must have that $W'/pW' \ra W/pW$ is surjective
and thus by Nakayama's Lemma that $W'=W$.
Since the $r$ elements $w_1,\dots, w_r$, generate the free rank $r$ $R$-module $W$, they must be a basis, and we have a dual basis $w_i^*$ of $W^*$.

Let $G^*$ be generated by $e_1^*, \dots ,e_r^* $ with relations $p^{e-\lambda_i}e_i^*$, 
and such that  $e_i^* e_i =p^{e-\lambda_i}$, and for $i\ne j$ we have $e_i^* e_j =0 $.

Recall we have
$$
s_F: \wedge^2 G^* \ra \Hom(V,G^*).
$$
We can take the further quotient 
$$
s'_F: \wedge^2 G^* \ra \Hom(W,G^*).
$$
We see that
$$
s_F(e_i^* \wedge e_j^*)(w_a)
= e_i^*(e_a) e_j^* - e_j^*(e_a) e_i^*.
$$
Recall that since $W$ is a free $R$-module, the natural map $W^* \tesnor G^* \ra \Hom(W,G^*)$ is an isomorphism.  
So we can determine that
$$
s'_F(e_i^* \wedge e_j^*)=p^{e-\lambda_i} w_i^* \tensor e_j^* -
p^{e-\lambda_j} w_j^* \tensor e_i^* .
$$
For $i<j$, this element has order $p^{\lambda_j}$, and we can easily conclude that
$$
p^{\lambda_2+2\lambda_2+\dots +(r-1)\lambda_r} \mid \#\im(s'_F) \mid \#\im(s_F).
$$
\end{proof}

Now we will see that as long as $C$ is not special (in particular even if it is weak), we can get a moderately good bound on the number of non-zero $E(C,F,i,j)$. 

\begin{lemma}[Linearly many non-zero coefficients for non-special $C$]\label{L:nonspec}
Given $F\in\Hom(V,G)$ a code of distance $\delta n$, suppose $C\in\Hom(V,G^*) \setminus \im(s_F)$ (so $C$ is not special for $F$).
Then there are at least $\delta/(2n)$ pairs $(i,j)$ with $i,j\in [n]$ and $i<j$ such that 
$$E(C,F,i,j) \ne 0 .$$
\end{lemma}
In other words, not only do we have $\im(s_F)=\ker (m_F),$ but in fact when $F$ is a code, we have that non-special $C$ are not even near $\ker (m_F).$

\begin{proof}
Suppose not, for contradiction.  Let $\sigma\sub[n]$ have $|\sigma|<\delta n$ and
for $(i,j)$ with $i \not \in \sigma $ or with $j \not \in \sigma$ we have
$$e(C(v_j) ,F(v_i) ) +e(C(v_i) ,F(v_j) ) =0 .$$

In Lemma~\ref{L:sfinj} we have a lower bound on the size of $\ker(m_F)$.
Next we will find a lower bound on the size of $\im(m_F)$.
Recall we have
\begin{align*}
m_F :\Hom(V,G^*) &\ra \Sym^2 V^*\\
C &\mapsto  \sum_{i=1}^n \sum_{j=i+1}^n (e(C(v_j) ,F(v_i) ) +e(C(v_i) ,F(v_j) )  )v_i^*v_j^* +\sum_{i=1}^{n} 
e(C(v_i) ,F(v_i) ) (v_i^*)^2.
\end{align*} 
We can take the further quotient using $\Sym^2 V^* \ra Z$ that sends
$v_i^*v_j^*$ to $0$ if $i,j\in \sigma$. 
Call this map
$$
m'_F:\Hom(V,G^*) \ra Z.
$$
So we have some $C$ which is not in $\im(s_F)$ but for which $m'_F(C)=0$.
We will show this is impossible by showing that $\#G^n/\#\wedge^2 G | \#\im(m'_F)$.
Once we have established  $\#G^n/\#\wedge^2 G | \#\im(m'_F)$, by combining with 
Lemma~\ref{L:sfinj}, we will see that $\im(s_f)=\ker(m'_F)$ and obtain a contradiction, proving the lemma.

As in the proof of Lemma~\ref{L:sfinj}, we can establish that $\#G^n/\#\wedge^2 G | \#\im(m'_F)$ by reducing the the case where $G$ is a $p$-group of type $\lambda$, which we will do for the rest of the proof of this lemma (and accordingly assume $R=\Z/p^e\Z$).

We can find $\tau \sub [n]\setminus \sigma$ such that  $|\tau|=r$ and
$Fv_i$ for $i\in \tau$ generate $G$ using the Lemma~\ref{L:FindAB}.
(Specifically, since $|\sigma|<\delta n$, and $F$ is a code of distance $\delta n$, we have
$FV_\sigma=G$ and so $F|_{V_\sigma}$ is a code of some positive distance.  We apply Lemma~\ref{L:FindAB} to $F|_{V_\sigma}$.)  Let $e_i$ be generators for $G$ with relations $p^{e-\lambda_i} e_i=0$. 
Let $G^*$ be generated by $e_1^*, \dots ,e_r^* $ with relations $p^{e-\lambda_i}e_i^*$, 
and such that  $e_i^* e_i =p^{e-\lambda_i}$ and for $i\ne j$, we have $e_i^* e_j =0 $.
As in the proof of Lemma~\ref{L:sfinj}, we can find an alternate basis $w_1,\dots,w_r$ for the
free $R$-module generated by the $v_i$ with $i\in\tau$,  with the property that that $Fw_i=e_i$.

We will in fact consider the further quotient by $\Sym^2 V^* \ra Z'$ that
sends $v_i^*v_j^*$  to $0$ for $i,j$ with neither $i$ nor $j$ in $\tau$.
Call this map
$$
m''_F:\Hom(V,G^*) \ra Z'.
$$

Note that $v_i^*$ for $i\not\in \tau$ and $w_i^*$ for $1\leq i\leq r$ form a basis of $V^*$.
We will call these $z_i^*$ for uniform notation.
In particular, denote $\tau=\{\tau_1,\dots,\tau_r\}$ and, for $\tau_i\in\tau$, let $z_{\tau_i}:=w_{i}.$


If we write
$F=\sum_{1\leq i\leq n, 1\leq j\leq r} f_{ij} z_i^* e_j$, we have
$$Fz_\ell=\sum_{ 1\leq k\leq r} f_{\ell k}  e_k.$$
Then
\begin{align*}
m_F(C)&=\sum_{i=1}^n \sum_{j=1, j\ne i}^n e(C(z_j) ,F(z_i) )z_i^*z_j^* +\sum_{i=1}^{n} 
e(C(z_i) ,F(z_i) ) (z_i^*)^2
\\
&=\sum_{i=1}^n \sum_{j=1, j\ne i}^n e(C(z_j) ,\sum_{ 1\leq k\leq r} f_{ik}  e_k )z_i^*z_j^* +\sum_{i=1}^{n} 
e(C(z_i) ,\sum_{ 1\leq k\leq r} f_{ik}  e_k ) (z_i^*)^2.
\end{align*}
So 
\begin{align*}
&m_F(z_\ell^*\tensor e_m^*)\\
&=\sum_{i=1}^n \sum_{j=1, j\ne i}^n e(z_\ell^*(z_j) \tensor e_m^* ,\sum_{ 1\leq k\leq r} f_{ik}  e_k ) z_i^*z_j^* +\sum_{i=1}^{n} 
e(z_\ell^*(z_i) \tensor e_m^* ,\sum_{ 1\leq k\leq r} f_{ik}  e_k ) (z_i^*)^2\\
&=\sum_{i=1}^n \sum_{j=1, j\ne i}^n f_{im}p^{e-\lambda_m} z_\ell^*(z_j)  (z_i^*z_j^*) +\sum_{i=1}^{n} 
f_{im}p^{e-\lambda_m} z_\ell^*(z_i)   (z_i^*)^2\\
&=\sum_{i=1}^n f_{im}p^{e-\lambda_m} z_i^*z_\ell^* .
\end{align*}

Since, $Fz_b=\sum_{ 1\leq k\leq r} f_{bk}  e_k,$
and $Fz_{\tau_i}=e_i,$ we have
$f_{\tau_{i} i}=1$ and $f_{\tau_{i}k}=0$ for $k\ne {i}$.
If $\ell \not \in \tau$,
\begin{align*}
&m''_F(z_\ell^*\tensor e_m^*)
=\sum_{i\in \tau} f_{im}p^{e-\lambda_m} z_i^*z_\ell^* =
\sum_{1\leq i \leq r} f_{\tau_i m}p^{e-\lambda_m} w_{i}^*z_\ell^*
=p^{e-\lambda_m} w_{m}^*z_\ell^*
.
\end{align*}
We see here that $\im( m''_F)$ has a subgroup of size $\#G^{n-r}.$
We can then form $m'''_F$, a further quotient to only terms $z_i^*z_j^*$ with $i,j,\in \tau$.
In particular, the subgroup of size $\#G^{n-r}$ we have identified above will go to $0$ under 
$m'''_F$.
If $\ell\in \tau$
\begin{align*}
&m'''_F(z_\ell^*\tensor e_m^*)
=\sum_{i\in \tau} f_{im}p^{e-\lambda_m} z_i^*z_\ell^* =
\sum_{1\leq i \leq r} f_{\tau_i\ell}p^{e-\lambda_\ell} w_{i}^*z_\ell^*
=p^{e-\lambda_m} w_{m}^*z_\ell^*
.
\end{align*}

So for $i\leq j$, we see that $p^{e-\lambda_i} w_{i}^* w_{j}^*\in \im( m'''_F)$.
It follows  that $\im( m'''_F)$ has a subgroup of size  $p^{\lambda_1r+\dots+\lambda_r}.$
We conclude that $\#G^n/\#\wedge^2 G =\#G^{n-r}p^{\lambda_1r+\dots+\lambda_r} \mid\im(m''_F)\mid\im(m'_F)$ .
This completes the proof of the lemma as explained above.

\end{proof}

\section{Obtaining the moments II: A good bound for surjections that are codes}\label{S:codes}

In this section, we put the results of the last section together to prove a good bound on the probability that a code descends to a map from the cokernel of a random matrix.


\begin{lemma}\label{L:FullFcode}
Let $0<\alpha<1$, and $\delta>0$, and $a$ a positive integer, and $G$ a finite abelian group of exponent dividing $a$.
Then there is a $c>0$
and a real number $K$ such that the following holds.

Let $X$ be a random symmetric $n \times n$ matrix, whose entries $X_{ij}$, for $i\leq j$, 
are independent.
Further, we require that for any prime $p|a$ and any $t\in \Z/p\Z$, the probability $\P(X_{ij} \equiv t \pmod{p})\leq 1-\alpha.$  
 Let $F\in \Hom(V,G)$ be a code of distance $\delta n$. Let $A\in\Hom(V^*,G)$.
   For all $n$ we have
\begin{align*}
 \left|\P(FX=0) - |\wedge^2 G||G|^{-n}\right| &\leq
 \frac{K\exp(-cn)}{|G|^{n}}.
\end{align*}
and
\begin{align*}
 \P(FX=A) \leq  K|G|^{-n}.
\end{align*}
\end{lemma}

From the point of view of descending a surjection to the cokernel of a matrix, we only need $A=0$ above, but in fact, for our work with non-codes we will need the case of general $A$ as above.

\begin{proof}
Recall,
$$
\P(FX=A) =\frac{1}{|G|^{n}}\sum_{C\in \Hom(V,G^*) } \E (\zeta^{C(FX-A)}),
$$
where $\zeta$ is a primitive $a$th root of unity.

We break the sum into $3$ pieces: (we will later choose $0<\dc<\delta$)
\begin{enumerate}
 \item when $C$ is special for $F$
\item when $C$ is not special for $F$ and is weak for $F$
\item when $C$ is robust for $F$.
\end{enumerate}

Given $F$, there are $|\wedge^2 G|$ special $C$ for which $\zeta^{C(FX)}=1$ for all $X$.
Thus, the sum from (1) contributes $|\wedge^2 G||G|^{-n}$ when $A=0$ and at most $|\wedge^2 G||G|^{-n}$ in absolute value for any $A$.

For (2), we will first use the fact that there are not too many weak $C$ and our bound from Lemma~\ref{L:nonspec}.
From Lemma~\ref{L:estweakC}, we have
that the number of $C\in \Hom(V,G^*)$ such that $C$ is weak for $F$ is at most
$$
C_G \binom{n}{\lceil \dc n \rceil -1} |G|^{\dc  n }.
$$

Next we factor the expected value
$$
\E (\zeta^{C(FX-A)}) =\E (\zeta^{C(-A)})\prod_{1\leq i < j\leq n} \E(\zeta ^{E(C,F,i,j)X_{ij}}  )\prod_{1\leq i \leq n} \E(\zeta ^{E(C,F,i,i)X_{ii}}  ).
$$
Let $u\in R=\Z/a\Z$ with $u\ne 0$.  Then we will show that $|\E(\zeta^{u X_{ij}})|\leq \exp(-\alpha /a^2)$. We rephrase the problem as follows. Let $b$ be an integer and $\xi\ne 1$ a $b$th root of unity.  For $t\in\Z/b\Z$ we have $0\leq p_{t}\leq 1-\alpha$ with $\sum_t p_t=1$.  Then we will see
$|\sum_t p_t\xi^t |\leq e^{-\alpha/b^2}.$ 
Let $U$ be a unit vector in the same direction as $E:=\sum_t p_t\xi^t $ in the complex plane.  
We consider the projections of the $\xi^t$ onto $U$ and their (signed) lengths $proj_U(\xi^t)$.
We have $E=U\sum_t p_t proj_U(\xi^t)$. 
 Let $c\in \Z/b\Z$ be so that among the $\xi^t$, the complex number $\xi^c$ is closest to $U$ in angle.  So for $t\ne c$, we have $proj_U(\xi^t)\leq \cos(\pi/b)$.  So 
 $$\sum_t p_t proj_U(\xi^t)\leq \cos(\pi/b) +p_c(1-\cos(\pi/b))\leq \alpha \cos(\pi/b) +1-\alpha.
$$
We have $-1+\cos(\pi/b) \leq -b^{-2}$ for $b\geq 1$.  
So $-1+\cos(\pi/b) \leq  -b^{-2} e^{-\alpha/b^2}.$ 
Integrating with respect to $\alpha$, we obtain 
$\alpha\cos(\pi/b) +1-\alpha\leq e^{-\alpha/b^2}$, and conclude $|\sum_t p_t\xi^t |\leq e^{-\alpha/b^2},$ as desired.

Given $F$ a code of distance $\delta n$ and a $C$ that is not special, by Lemma~\ref{L:nonspec} we have that at least $\delta n/2$ of the $E(F,C,i,j)$ are non-zero. 
 So if $C$ is not special for $F$, we conclude that
$$
|\E (\zeta^{C(FX-A)})| \leq \exp(-\alpha \delta n/(2a^2)).
$$
 
Now, given $F$ and a robust $C$ for $F$, by Lemma~\ref{L:Quad}, 
we have that at least $\dc \delta n^2/(2|G|^2|P|)$ of the $E(C,F,i,j)$ are non-zero (where $P$ is the set of primes dividing $a$). 
 So if $C$ is robust for $F$, we conclude that
$$ 
|\E (\zeta^{C(FX-A)})| \leq \exp(-\alpha \dc \delta n^2/(2|G|^2|P|a^2)).
$$
 
In conclusion
\begin{align*}
 &\left|\P(FX=A) - \frac{1}{|G|^{n}} \sum_{C \in \Hom(V,G^*), \textrm{ special}}|\E (\zeta^{C(FX-A)})\right| \\
&\leq
 \frac{1}{|G|^{n}} \sum_{C \in \Hom(V,G^*), \textrm{ not special}}| \E (\zeta^{C(FX-A)})|\\
 &\leq
 \frac{1}{|G|^{n}} 
\left( 
 C_G \binom{n}{\lceil \dc n \rceil -1} |G|^{\dc  n }
 \exp(-\alpha \delta n/(2a^2))+ |G|^{n} \exp(-\alpha \dc \delta n^2/(2|G|^2|P|a^2))
 \right).
\end{align*}
So for any $c>0$ such that $c< \alpha \delta /(2a^2)$, given
 given $\delta, \alpha, G,c$, we can choose $\dc$ sufficiently small so that we have
\begin{align*}
& \left|\P(FX=A) -  \frac{1}{|G|^{n}} \sum_{C \in \Hom(V,G^*), \textrm{ special}} \E (\zeta^{C(FX-A)})\right| \\&\leq
 \frac{1}{|G|^{n}} \left(C_G \exp(-cn)  + \exp(\log(|G|)n-\alpha \dc \delta n^2/(2|G|^2|P|a^2)) \right).
\end{align*}
For $n$ sufficiently large given $\alpha,G,\delta,c,\dc$, we have $$\log(|G|)n-\alpha \dc \delta n^2/(2|G|^2|P|a^2)\leq -cn.$$
So in the case $A=0$, for $n$ sufficiently large, we have 
\begin{align*}
 \left|\P(FX=0) - |\wedge^2 G ||G|^{-n}\right| &\leq
 \frac{(C_G+1)\exp(-cn)}{|G|^{n}}.
\end{align*}
For $n$ that aren't sufficiently large, we will just increase the constant $K$ in the lemma.

For any $A$, we have for $n$ sufficiently large given $\alpha,|G|,\delta,c,\dc$,
\begin{align*}
 \P(FX=A) &\leq \left|  \frac{1}{|G|^{n}} \sum_{C \in \Hom(V,G^*), \textrm{ special}} \E (\zeta^{C(FX-A)})\right| +
\frac{(C_G+1)\exp(-cn)}{|G|^{n}}
\\
&\leq |G|^{-n}(|\wedge^2 G | + C_G +1).
\end{align*}
For $n$ that aren't sufficiently large, we can increase the constant $K$ as necessary.
\end{proof}

\section{Obtaining the moments III: Determining the structural properties of the surjections}\label{S:depth}

In the last section, we dealt with $F\in\Hom(V,G)$ that were codes.  Unfortunately, it is not sufficient to divide $F$ into codes and non-codes.  We need a more delicate division of $F$ based on the subgroups of $G$.

For an integer $D$ with prime factorization $\prod_i p_i^{e_i}$, let $\ell(D)=\sum_i e_i$.

\begin{definition}
The \emph{depth} of an $F\in\Hom(V,G)$ is the maximal positive $D$ such that
there is a $\sigma\sub [n]$ with $|\sigma|< \ell(D)\delta n$ such that $D=[G:FV_\sigma]$, or is $1$ if there is no such $D$. 
\end{definition}

\begin{remark}\label{R:depthcode}
In particular, if the depth of $F$ is $1$, then for every $\sigma\sub [n]$ with $|\sigma|< \delta n$,
we have that $FV_\sigma=G$ (as otherwise $\ell([G:FV_\sigma])\geq 1$), and so we see that $F$ is a code of distance $\delta n$.
\end{remark}
 Also, if the depth of $F$ is $D$, then $D\mid\#G$.
Now we will bound the number of $F$ that we have of depth $D$.

\begin{lemma}[Count $F$ of given depth]\label{L:countdepth}
There is a constant $K$ depending on $G$ such that if $D>1$, then number of $F\in \Hom(V,G)$ of depth $D$ is at most
$$
K\binom{n}{\lceil \ell(D)\delta n \rceil -1} |G|^n|D|^{-n+\ell(D)\delta n}.
$$

\end{lemma}
\begin{proof}
We sum over $\sigma\sub [n]$ with $\sigma=\lceil \ell(D)\delta n \rceil -1 $ the number of $F$ such that $D=[G:FV_\sigma]$.  Then we sum over the subgroups of $G$ of index $D$ (this sum will go into the constant.)
Now given a particular subgroup $H$ of index $D$, we bound the number of $F$ such that $FV_\sigma=H.$ 
We have at most
$(|G|/D)^{(n-|\sigma|)}$ maps from $V_\sigma$ to $H$, and at most $|G|^{|\sigma|}$ choices for the $Fv_i$ with $i\in \sigma$.
So, for a particular $\sigma$ and $H$,  the number of $F$ such that $FV_\sigma=H$  is at most
$$
(|G|/D)^{(n-|\sigma|)}|G|^{|\sigma|}=|G|^n D^{-n+|\sigma|}.
$$
Note that $|\sigma|<\ell(D)\delta n,$ and the lemma follows.
\end{proof}

The following is a variant on the bound on the number of $F$ of depth $D$ that we will need when we are working with the Laplacian of a random graph.

\begin{lemma}\label{L:countdepthgraph}
Let $pr_2: G \oplus R \ra R$ be the projection onto the second factor.
There is a constant $K$ depending on $G$ such that if $D> 1$, then number of $F\in \Hom(V,G\oplus R)$
such that $pr_2(Fv_i)=1$ for all $i\in [n]$, and
 of depth $D$ is at most
$$
K\binom{n}{\lceil \ell(D)\delta n \rceil -1} |G|^n|D|^{-n+\ell(D)\delta n}.
$$

\end{lemma}
\begin{proof}
Let $G'=G\oplus R$.  Note that $|G'|=a|G|.$
We sum over $\sigma\sub [n]$ with $\sigma=\lceil \ell(D)\delta n \rceil -1 $ the number of $F$ such that
$FV_\sigma$ has index $D$ in $G'$.  Then we sum over the subgroups of $G'$ of index $D$ (this sum will go into the constant.)
Now given a particular subgroup $H$ of index $D$, we bound the number of $F$ such that $FV_\sigma=H.$ 

For $i \in [n]\setminus\sigma$, we must have $Fv_i\in H$ and $pr_2(F v_i) =1$.    There are at most $H/a$ elements $h\in H$ such that $pr_2( h)=1.$  
So there are at most
$$
(|H|/a)^{n-|\sigma|}=(|G'|/(aD))^{n-|\sigma|}=(G/D)^{n-|\sigma|}
$$
possibilities for $F|_{V_\sigma}.$
There are at most $|G|^{|\sigma|}$ choices for the $Fv_i$ with $i\in \sigma$.
So, for a particular $\sigma$ and $H$,  the number of $F$ such that $FV_\sigma=H$ (and $pr_2(Fv_i)=1$ for all $i\in [n]$)  is at most
$$
(G/D)^{n-|\sigma|} |G|^{|\sigma|}= |G|^n D^{-n+|\sigma|}.
$$
Note that $|\sigma|<\ell(D)\delta n,$ and the lemma follows.
\end{proof}

For each depth $D$ of $F$, we will use the following specially tailored bound for $\P(FX=0).$

\begin{lemma}[Bound probability given depth]\label{L:probdepthestimate}
Let $\alpha,\delta,G,a$ be as in Lemma~\ref{L:FullFcode}. Then there is a real $K$ such that if $F\in\Hom(V,G)$ has depth $D> 1$ and  $[G:FV]<D$ (e.g. the latter is true of $FV=G$), then for all $X$
as in Lemma~\ref{L:FullFcode}
and all $n$,
$$\P(FX=0)\leq 
K e^{-\alpha(1-\ell(D)\delta)n} (|G|/D)^{-(1-\ell(D)\delta)n}
.$$
\end{lemma}

\begin{proof}
Pick a $\sigma\sub [n]$ with $|\sigma|< \ell(D)\delta n$ such that $D=[G:FV_\sigma].$
Let $FV_\sigma=H.$  

We will now divide the elements $i\in [n]$ depending on whether $Fv_i\in H$.  Let $\eta$ be the set of $i$ such that
$Fv_i\in H$.  Let $\tau=[n]\setminus \eta$.  Note that $[n]\setminus \sigma \sub \eta$, so
$\tau\sub \sigma,$ and so $|\tau|<\ell(D) \delta n$.
However, since $[G:FV]<D$, we cannot have $\tau$ empty. 


 Let $X_{\eta}$ be obtained from $X$ by replacing the entries in the ${ \eta}$ rows with $0$, and let $X_\tau$ be obtained from $X$ by replacing the entries in the $\tau$ rows with $0$.  
 Note that $FX\in \Hom(V^*, G).$  We identify $\Hom(V^*, G)$ with $G^n$ using the preferred basis of $V^*$.
We have
$$
\P(FX=0)=\P(FX_\tau \in H^n) \P(FX=0 | FX_\tau \in H^n).
$$

Let $X_{\tau\eta}$ be obtained from $X_\tau$ be replacing the entries in the $\tau$ columns with $0$.
Note that all the entries in $X_{\tau\eta}$ that are not forced to be zero are independent.  
We have 
$$\P(FX_\tau \in H^n) \leq \P(FX_{\tau\eta} \in H^n).$$

Note that all the entries in $X_{\tau\eta}$ that we have not made zero are independent.  
So
$$ \P(FX_{\tau\eta} \in H^n)=\prod_{i\in \eta}\P(Fcol_i(X_{\tau}) \in H).
$$
Consider a single column and let $x_1,\dots,x_{|\tau|}$ be the entries in the $\tau$ rows of $X$, and $f_1,\dots,f_{|\tau|}\in G\setminus H$ be the corresponding entries of $F$.  We condition on $x_2,\dots,x_{|\tau|}$.  Then, for some fixed $g\in G$, and $f_1\in G\setminus H$, we are trying to bound 
$$
\P(f_1x_1 \equiv g \textrm{ in $G/H$} ).
$$
Since $f_1\not\equiv 0 \pmod{G/H}$, there is some prime $p$ that divides the order of $f_1$ in $G/H$.
Note that if $f_1x \equiv g \textrm{ in $G/H$}$, then for $\Delta\in \Z$ such that $p\nmid \Delta$, we have
$f_1(x+\Delta) \not \equiv g \textrm{ in $G/H$}$.  So the $x$ such that $f_1x \equiv g \textrm{ in $G/H$}$ are contained in a single equivalence class modulo $p$.  
Thus,
$$
\P(f_1x_1 \equiv a \textrm{ in $G/H$} )\leq e^{-\alpha},
$$
since the probability that $x_1$ is any any particular equivalence class mod $p$ is at most $e^{-\alpha}$.
We can then conclude
$$\P(FX_\tau \in H^n) 
 \leq e^{-\alpha|\eta|}.$$

Now let $X_{\eta \eta}$ be obtained from $X_{\eta}$ by replacing the entries in $\tau$ columns with $0$.
Let $X_{*\eta}=X_{\eta \eta}+X_{\tau\eta}$ (so $X_{*\eta}$ is obtained from $X$ by replacing the $\tau$ columns with $0$).
We have 
$$
\P(FX=0 | FX_\tau \in H^n)\leq \P(FX_{*\eta}=0 | FX_\tau \in H^n).
$$
We estimate $\P(FX_{*\eta}=0 | FX_\tau \in H^n)$ by conditioning on the $\tau$ rows (and columns) of $X$.
  Then for any $n\times n$ matrix $Y_\tau$ over $R$ supported on the $\tau$ rows  with  $FY_\tau \in H^n$ (and with 
  $Y_{\tau\eta}$ obtained from $Y_\tau$ by replacing the entries in the $\tau$ columns by $0$),   

  $$\P(FX_{*\eta}=0 | X_\tau=Y_\tau)=
  \P(FX_{\eta\eta}+FY_{\tau\eta}=0 | X_\tau=Y_\tau).
  $$
In particular, $FY_{\tau\eta}$ is some fixed value in $H^{|\eta|}.$  
Also, note that $X_{\eta\eta}$ is independent of $X_\tau$.
For a fixed $A \in H^{|\eta|}$, we need to estimate $
\P(FX_{\eta\eta}=A).
$
Note that $F|_{V_\tau}$ (i.e. restricted to the $\eta$ indices) is a code of distance $\delta n$ in $\Hom(V_\tau, H)$.
(If it were not, then by eliminating $\tau$ and $<\delta n$ indices, we would eliminate $<(\ell(D)+1)\delta n$ indices
and have an image which was index that $D$ strictly divides, contradicting the depth of $F$.)
So by Lemma~\ref{L:FullFcode}, for some $K$
$$
  \P(FX_{\eta\eta}+FY_{\tau\eta}=0 | X_\tau=Y_\tau)=
 \P(FX_{\eta\eta}=-FY_{\tau\eta} \in H^{|\eta|} )  
  \leq K|H|^{-|\eta|}.
$$
So we conclude,
$$
\P(FX=0)\leq K e^{-\alpha|\eta|} |H|^{-|\eta|} \leq K_1 e^{-\alpha(1-\ell(D)\delta)n} (|G|/D)^{-(1-\ell(D)\delta)n}.
$$
\end{proof}

\section{Obtaining the moments IV: Putting it all together}\label{S:FinalMom}

We can combine our work above to give a universality result on (and the actual values of) the moments of cokernels of random matrices and of sandpile groups of random graphs, which we do in the following two theorems, respectively.  
\begin{theorem}\label{T:MomMat}
Let $0<\alpha<1$ be a real number and $G$ a finite abelian group.
For any $c<\min(\alpha, \log(2))$, there is a $K>0$ (depending on $\alpha, G, c$) such that the following holds.
Let $X$ be a random symmetric $n \times n$ matrix, whose entries $X_{ij}\in\Z$, for $i\leq j$, 
are independent.
Further, we require that for any prime $p|G$ and any $t\in \Z/p\Z$, the probability $\P(X_{ij} \equiv t \pmod{p})\leq 1-\alpha.$  Then,
\begin{align*}
\left|\E(\#\Sur(\cok(X),G))  -|\wedge^2 G| \right| \leq Ke^{-cn}.
\end{align*}
\end{theorem}
\begin{proof}
We omit the details in this proof, as they are almost identical to (and slightly simpler than) the details in the proof of our next result, Theorem~\ref{T:MomGraphs}, which is in the case of main interest.
If the exponent of $G$ divides $a$, we can reduce $X$ modulo $a$ so as to agree with our notation above.
We wish to estimate
$
\sum_{F\in \Sur(V,G)} \P(FX=0).
$
Using Lemmas~\ref{L:countdepth} and \ref{L:probdepthestimate} we have
\begin{align*}
\sum_{\substack{F\in \Sur(V,G)\\ F\textrm{ not  code of distance $\delta n$}
} }
\P(FX=0)\leq  K 
 e^{-cn}.
\end{align*}
Also, from Lemma~\ref{L:countdepth}
\begin{align*}
\sum_{\substack{F\in \Sur(V,G)\\ F\textrm{ not  code of distance $\delta n$}
} }
|\wedge^2 G||G|^{-n}\leq  K 
 e^{-cn}.
\end{align*}
We also have 
 \begin{align*}
\sum_{\substack{F\in \Hom(V,G)\setminus \Sur(V,G)
} }
|\wedge^2 G||G|^{-n}\leq K 2^{-n}.
\end{align*}
Then we have, using Lemma~\ref{L:FullFcode},
\begin{align*}
\sum_{\substack{F\in \Sur(V,G)\\ F\textrm{  code of distance $\delta n$}
} }
\left| \P(FX=0) -|\wedge^2 G||G|^{-n}\right|
&\leq 
 Ke^{-cn} . 
\end{align*}
Combining, we obtain the theorem.

\end{proof}

In particular, the following theorem implies Theorem~\ref{T:Mom}.
\begin{theorem}\label{T:MomGraphs}
Let $0<q<1$ and let $G$ be a finite abelian group.  Then there exist  $c,K>0$ such that if 
 $\Gamma\in G(n,q)$ is a random graph,  $S$ is its sandpile group, for all $n$ we have
\begin{align*}
& \left| \E(\#\Sur(S,G)) -|\wedge^2 G| \right| \leq Ke^{-cn}.
\end{align*}

\end{theorem}
\begin{proof}
Let $a$ be the exponent of $G$.  Let $R=\Z/a\Z$.  
Note that $\#\Sur(S,G)=\#\Sur(S\tensor R,G)$, so throughout this proof we will let $\bS:=S\tensor R$.
We let $\bL$ be the reduction of the Laplacian $L$ modulo $a$, so $\bL$ is an $n\times n$ matrix with coefficients in $R$. 

We let $X$ be an $n \times n$ random symmetric matrix with coefficients in $R$ with $X_{ij}$ distributed as $\bL_{ij}$ for $i<j$ and with $X_{ii}$ distributed uniformly in $R$, with all $X_{ij}$ (for $i<j$) and $X_{ii}$ independent.
Let $F_0\in \Hom(V,R)$ be the map that sends each $v_i$ to $1$.
If we condition on $F_0 X=0$, then we find that $X$ and $\bL$ have the same distribution.
In particular, given $X$ and conditioning on the off diagonal entries, we see that the probability that $F_0 X=0$ is $a^{-n}$ (for any choice of off diagonal entries).
So any choice of off diagonal entries is equally likely in $\bL$ as in $X$ conditioned on $F_0 X=0$.

Recall $V=R^n$.
So for $F\in \Hom(V,G)$, we have
\begin{align*}
\P(F\bL=0)&=\P(FX=0 | F_0 X=0)
=\P(FX=0 \textrm{ and } F_0 X=0)a^{n}.
\end{align*}
Let $\tilde{F}\in\Hom(V, G\oplus R)$ be the sum of $F$ and $F_0$.

Let $Z\sub V$ denote the vectors whose coordinates sum to $0$, i.e. $Z=\{v\in V \ |\ F_0v=0\}$.
Let $\Sur^* (V,G)$ denote the maps from $V$ to $G$ that are a surjection when restricted to $Z$.
We wish to estimate
\begin{align*}
\E(\#\Sur(\bS,G))&=\E(\#\Sur(Z/\cs(\bL),G))\\
&=\sum_{F\in \Sur(Z,G)}   \P(F\bL=0)\\
&=\frac{1}{|G|}\sum_{F\in \Sur^*(V,G)} \P(F\bL=0) \\&={|G|^{-1}a^n}\sum_{F\in \Sur^*(V,G)} \P(\tilde{F}X=0).
\end{align*}
Note that if $F: V \ra G$ is a surjection when restricted to $Z$, then $\tilde{F}$ is a surjection from $V$ to $G\oplus R$.

We start by considering the part of the sum to which we can apply Lemma~\ref{L:probdepthestimate}.
We let $K$ change in each line, as long as it is a constant depending only on $q,G,\delta$.
Let $\alpha=\max(q,1-q)$.
We then have
\begin{align*}
&\frac{a^n}{|G|}\sum_{\substack{F\in \Sur^*(V,G)\\ \tilde{F}\textrm{ not  code of distance $\delta n$}
} }
\P(\tilde{F}X=0)\\
&\leq 
\frac{a^n}{|G|}\sum_{\substack{D>1\\ D\mid\#G}} \sum_{\substack{F\in \Sur^*(V,G)\\ \tilde{F}\textrm{  depth $D$}}}
\P(\tilde{F}X=0) \quad\quad \textrm{(by Remark~\ref{R:depthcode})}\\
&\leq\frac{a^n}{|G|}\sum_{\substack{D>1\\ D\mid\#G}} 
\#\{ \tilde{F}\in
\Hom(V,G\oplus R) \textrm{ depth } D\ | \ pr_2(v_i)=1 \textrm{ for all }i 
\}
 K  e^{-\alpha(1-\ell(D)\delta)n} (a|G|/D)^{-(1-\ell(D)\delta)n}\\
&\leq \frac{a^n}{|G|}\sum_{\substack{D>1\\ D\mid\#G}} 
K\binom{n}{\lceil \ell(D)\delta n \rceil -1} |G|^nD^{-n+\ell(D)\delta n}
  e^{-\alpha(1-\ell(D)\delta)n} (a|G|/D)^{-(1-\ell(D)\delta)n} \quad\quad \textrm{(by Lemma~\ref{L:countdepthgraph})}\\
    &\leq
K\binom{n}{\lceil \ell(|G|)\delta n \rceil -1} 
  e^{-\alpha(1-\ell(|G|)\delta)n} (a|G|)^{\delta\ell(|G|)n}\\
 &\leq  K 
 e^{-cn}.
\end{align*}
For any $0<c<\alpha$, we can choose $\delta$ small enough so that 
$$\binom{n}{\lceil \ell(|G|)\delta n \rceil -1} 
  e^{-\alpha(1-\ell(|G|)\delta)n} (a|G|)^{\delta\ell(|G|)n} \leq e^{-cn} ,$$
 and the last inequality in the long chain above holds.

Also,
\begin{align*}
\sum_{\substack{F\in \Sur^*(V,G)\\ \tilde{F}\textrm{ not  code of distance $\delta n$}
} }
|\wedge^2 G||G|^{-n}&\leq 
\sum_{\substack{D>1\\ D\mid\#G}} \sum_{\substack{F\in \Sur^*(V,G)\\ \tilde{F}\textrm{  depth $D$}
} }
|\wedge^2 G||G|^{-n}\\
\textrm{(by Lemma~\ref{L:countdepthgraph})}
&\leq \sum_{\substack{D>1\\ D\mid\#G}}  K\binom{n}{\lceil \ell(D)\delta n \rceil -1} |G|^n|D|^{-n+\ell(D)\delta n}|\wedge^2 G||G|^{-n} 
\\
&\leq K \binom{n}{\lceil \ell(|G|)\delta n \rceil -1}2^{-n+\ell(|G|)\delta n}  \\
 &\leq  K 
 e^{-cn}.
\end{align*}
For any $0<c<\log(2)$, we can choose $\delta$ small enough so that 
$\binom{n}{\lceil \ell(|G|)\delta n \rceil -1}2^{-n+\ell(|G|)\delta n} \leq e^{-cn} ,$
 and the last inequality above holds. 

We also have 
 \begin{align*}
\sum_{\substack{F\in \Hom(V,G)\setminus \Sur^*(V,G)
} }
|\wedge^2 G||G|^{-n}&\leq 
\sum_{H \textrm{ proper s.g of } G} \sum_{\substack{F\in \Hom(Z,H)\
} }
|\wedge^2 G||G|^{-n+1}\\
&\leq 
\sum_{H \textrm{ proper s.g of } G} |H|^{n-1}
|\wedge^2 G||G|^{-n+1}\\
&\leq K 2^{-n}.
\end{align*}

Then we have, using Lemma~\ref{L:FullFcode},
\begin{align*}
\sum_{\substack{F\in \Sur^*(V,G)\\ \tilde{F}\textrm{  code of distance $\delta n$}
} }
\left| \P(\tilde{F}X=0) -|\wedge^2 (G\oplus R)|(a|G|)^{-n}\right|
&\leq 
 \sum_{\substack{F\in \Sur^*(V,G)\\ \tilde{F}\textrm{  code of distance $\delta n$}
} } Ke^{-cn}  (a|G|)^{-n}\\
&\leq 
 Ke^{-cn}a^{-n} . 
\end{align*}

In conclusion
\begin{align*}
& \left| \frac{a^n}{|G|}\left(\sum_{F\in \Sur^*(V,G)} \P(\tilde{F}X=0) \right) -|\wedge^2 G| \right| \\
&\leq
\left| \frac{{a^n}}{|G|}\sum_{\substack{F\in \Sur^*(V,G)\\ \tilde{F}\textrm{ not  code of distance $\delta n$}
} }
\P(\tilde{F}X=0) \right| +
\frac{a^n}{|G|}\sum_{\substack{F\in \Sur^*(V,G)\\ \tilde{F}\textrm{  code of distance $\delta n$}
} }
\left| \P(\tilde{F}X=0) -|\wedge^2(G\oplus R)|(a|G|)^{-n}\right|\\
&+\left| -|\wedge^2 G|+
\sum_{\substack{F\in \Sur^*(V,G)\\ \tilde{F}\textrm{  code of distance $\delta n$} 
} } |\wedge^2(G\oplus R)|(a|G|)^{-n}
\right|\\
&\leq
 Ke^{-cn}+
\sum_{\substack{F\in \Sur^*(V,G), \\ \tilde{F}\textrm{ not code of distance $\delta n$} 
} } |\wedge^2 G||G|^{-n}
+
\sum_{\substack{F\in\Hom(V,G)\setminus \Sur^*(V,G)
} } |\wedge^2 G||G|^{-n}
\\
&\leq Ke^{-cn}.
\end{align*}

Recall from above, we have 
\begin{align*}
\E(\#\Sur(S,G))={|G|^{-1}a^n}\sum_{F\in \Sur^*(V,G)} \P(\tilde{F}X=0),
\end{align*}
and so we conclude the proof of the theorem.

\end{proof}

\section{Basic estimates on abelian groups}\label{S:Abelian}
In this section we collect some basic estimates on numbers of maps between  finite abelian groups.  We do not claim any originality of the results in this section, but we merely collect them here for completeness.
We write $G_\lambda$ for the abelian $p$-group of type $\lambda$.

\begin{lemma}\label{L:Hom}
 We have $$ |\Hom(G_\mu,G_\lambda)|=p^{\sum_i\mu'_i \lambda'_i }.$$
\end{lemma}
\begin{proof}
A generator of $G_\mu$ of order $p^k$, can map to any element of $G_\lambda$ of order dividing $p^k$.
We have that $\Z/p^{\lambda_i}\Z$ has $p^{\min (k, \lambda_i)}$ elements of order dividing $p^k$.
So, we have $p^{\sum_i \min (k, \lambda_i)}$ that the generator can map to.  Note that 
$\sum_i \min (k, \lambda_i)=\sum_{j=1}^k \lambda_j'.$
So we have
\begin{align*}
  |\Hom(G_\mu,G_\lambda)| = p^{\sum_i  \sum_{j=1}^{\mu_i }\lambda_j' }.
\end{align*}

In the sum $
\sum_i  \sum_{j=1}^{\mu_i }\lambda_j' 
$, we consider the coefficient of $\lambda'_k$.  The coefficient of $\lambda'_k$
is the number of the $\mu_i$ such that $\mu_i\geq k,$ which is $\mu'_k$.  So $
\sum_i  \sum_{j=1}^{\mu_i }\lambda_j' =\sum_i \mu_i'\lambda_i'.$
\end{proof}

\begin{lemma}\label{L:Aut}
We have
$$
(\prod_{i\geq 1} (1-2^{-i}))^{\lambda_1}p^{\sum_i(\lambda'_i)^2 } \leq |\Aut(G_\lambda)| \leq p^{\sum_i(\lambda'_i)^2  }.
$$
\end{lemma}
\begin{proof}
 The right inequality follows from Lemma~\ref{L:Hom}.  
Let $m_i:=\lambda'_i-\lambda'_{i+1}$, so if $G_i=(\Z/p^i\Z)^{m_i}$, then $G_\lambda=\bigoplus_{i} G_i$.
We claim
$$
|\Aut(G_\lambda)|=p^{\sum_i (\lambda'_i)^2 } \prod_{i=1}^{\lambda_1}  \prod_{j=1}^{m_i} (1-p^{-j}).
$$ 
This is a standard fact (see e.g. \cite[Theorem 4.1]{HR07}), but we point out some observations that make it simple to prove.
By Nakayama's Lemma, a homomorphism $\phi: G_\lambda \ra G_\lambda$ is an isomorphism if and only if it is an isomorphism modulo $p$.  Modulo $p$, the summand $G_i$ can only map nontrivially to $G_j$ for $j\leq i$.  
Thus it follows that $\phi$ is an isomorphism if and only if it gives an isomorphism  $G_i/pG_i \ra G_i/pG_i$ for each $k$.

Similarly, and using that fact that 
$
\#\{ \textrm{symmetric matrices in $\GL_n(\Z/p\Z)$}\}=
p^{\frac{n(n+1)}{2}} \prod_{j=1}^{\lceil n/2 \rceil} (1-p^{1-2j})
$ (from \cite[Theorem2]{MacWilliams1969}), we can prove
$$
\#\{\textrm{symmetric, bilinear, perfect $G\times G
\ra \C^*$} \}=p^{\sum_i \lambda'_i(\lambda'_i+1)/2 }\prod_{i=1}^{\lambda_1}\prod_{j=1}^{\lceil m_i/2 \rceil} (1-p^{1-2j}) ,
$$ 
which is used in Equation~\eqref{E:main}.

\end{proof}

\begin{lemma}\label{L:Nsub}
 Let $\mu$ and $\lambda$ be partitions.  Let $G_{\mu,\lambda}$ be the set of subgroups of $G_\lambda$ that are isomorphic to $G_\mu$.  Then
$$
| G_{\mu,\lambda}| \leq \frac{1}{(\prod_{i\geq 1} (1-2^{-i}))^{\lambda_1}} p^{\sum_{i=1}^{\lambda_1}\mu'_i \lambda'_i  -(\mu'_i) ^2  }.
$$
\end{lemma}
\begin{proof}
Let $\operatorname{Inj}(G_\mu,G_\lambda)$ denote the set of injections from $G_\mu$ to $G_\lambda$.
We have  a map
 $$\operatorname{Inj}(G_\mu,G_\lambda) \ra G_{\mu,\lambda}$$
taking an injection to its image.  The group $\Aut(G_\mu)$ acts simply transitively on the fibers of this map,
so
$$
| G_{\mu,\lambda}| = \frac{|\operatorname{Inj}(G_\mu,G_\lambda)|}{|\Aut(G_\mu)|}\leq
\frac{|\Hom(G_\mu,G_\lambda)|}{|\Aut(G_\mu)|}.
$$ 
Note if $\mu_1>\lambda_1$, then $G_{\mu,\lambda}=0$.
When $\mu_1\leq \lambda_1$, we apply Lemmas~\ref{L:Hom} and \ref{L:Aut} to obtain the result.
\end{proof}

\begin{lemma}\label{L:BoundMom}
Let $G_\lambda$ be an abelian $p$-group of type $\lambda$.
Let $F=\frac{2}{1-2^{-1/8}}\prod_{i\geq 1} (1-2^{-i})^{-1}.$ 
$$
\sum_{G_1 \textrm{ subgroup of } G } |\wedge^2 G_1|\leq \prod_{j=1}^s {F^{\lambda_1}p^{\sum_i \frac{\lambda'_i(\lambda'_i-1)}{2}}}.
$$
\end{lemma}

\begin{proof}
 We have  
$$
\sum_{G_1 \textrm{ subgroup of } G } |\wedge^2 G_1| =\sum_\mu |G_{\mu,\lambda} |p^{\sum_j \frac{\mu'_j(\mu'_j-1)}{2}}.
$$
Note that we only have to sum over $\mu$ that are subpartitions of $\lambda$, or else $| G_{\mu,\lambda} |=0$.
In particular, we only have to sum over $\mu$ such that $\mu_1\leq \lambda_1$.

Let $C:=\prod_{i\geq 1} (1-2^{-i})$ and $D:=\frac{2}{1-2^{-1/8}}$.
We apply Lemma~\ref{L:Nsub},
\begin{align*}
\sum_\mu | G_{\mu,\lambda}| p^{\sum_j \frac{\mu'_j(\mu'_j-1)}{2}}
&\leq \frac{1}{C^{\lambda_1}} \sum_{\mu, \mu_1\leq \lambda_1}  p^{\sum_{i=1}^{\lambda_1} \mu'_i 
\lambda'_i  -(\mu'_i) ^2 + \frac{\mu'_i(\mu'_i-1)}{2}}\\
&= 
\frac{1}{C^{\lambda_1}} \sum_{d_1,\dots,d_{\lambda_1}\geq 0}  p^{\sum_{i=1}^{\lambda_1} d_i \lambda'_i  -d_i ^2 + \frac{d_i(d_i-1)}{2}}\\
&= 
\frac{p^{\sum_i\frac{\lambda'_i(\lambda'_i-1)}{2}}}{C^{\lambda_1}} \sum_{d_1,\dots,d_{\lambda_1}\geq 0}  (p^{\frac{1}{8}})^{\sum_{i=1}^{\lambda_1} -(2d_i - 2\lambda_i' +1)^2 +1}\\
&\leq 
\frac{p^{\sum_i\frac{\lambda'_i(\lambda'_i-1)}{2}}}{C^{\lambda_1}} \sum_{e_1,\dots,e_{\lambda_1}\in \Z,\ e_i \textrm{ odd}}  (p^{\frac{1}{8}})^{\sum_{i=1}^{\lambda_1} -e_i^2 +1 }.
\end{align*}
We have that
$$
\sum_{e_1\in \Z,\ e_1 \textrm{ odd}} (p^{\frac{1}{8}})^{ -e_1^2 +1}\leq 2 \sum_{e_1 \geq 1} (p^{\frac{1}{8}})^{ -e_1^2 +1}
\leq 2 \sum_{e_1 \geq 1} (p^{\frac{1}{8}})^{ -e_1 +1 }=\frac{2}{1-p^{-1/8}}\leq D.
$$
So, applying this to each of the $\lambda_1$ sums, we have
\begin{align*}
\sum_{G_1 \textrm{ subgroup of } G } |\wedge^2 G_1| 
&\leq \frac{p^{\sum_i\frac{\lambda'_i(\lambda'_i-1)}{2}}}{C^{\lambda_1}} \sum_{e_1,\dots,e_{\lambda_1}\in \Z,\ e_i \textrm{ odd}}  (p^{\frac{1}{8}})^{\sum_{i=1}^{\lambda_1} -e_i^2 +1 }
&\leq 
\frac{p^{\sum_i \frac{\lambda'_i(\lambda'_i-1)}{2}}}{C^{\lambda_1}} D^{\lambda_1}.
\end{align*}

\end{proof}

\section{Moments determine the distribution}\label{S:MomDet}

In this section we will see that the moments we have found in fact determine the distributions of our group valued random variables.
We have been working with the moments $\E(\# \Sur(-,G))$ so far.  As we have seen, these take nice values.  Summing over all subgroups of $G$, we then obtain the moments $\E(\# \Hom(-,G)).$  From the point of view of how we found these moments, the ``Hom'' moments are just derivative from the ``Sur'' moments.  However, from an analytic point of view, the moments $\E(\# \Hom(-,G))$ are easier to work with.  For example, if $G=(\Z/p\Z)^k$ for a prime $p$, then $\# \Hom(H,G)=p^{(p\textrm{-rank}(H))k}$.  So for $G=(\Z/p\Z)^k$, the Hom moments give the usual moments of the random variable $p^{p\textrm{-rank}(H)}$.  
For example, when $S$ is a sandpile group of a random graph, as above, we have that $\E(\left(p^{p\textrm{-rank}(S)}\right)^k)$ is on the order of $p^{k(k-1)/2}$ (see Lemma~\ref{L:BoundMom}).

As discussed in the introduction, these moments are too big to use Carleman's condition to recover the distribution, but we can take advantage of the fact that we know the random variable takes only values $p^i$, where $i$ is a non-negative integer. 
In this case, we can view the problem as one of a countably infinite system of linear equations which we would like to show has a unique (non-negative) solution. 
Heath-Brown \cite[Lemma 17]{Heath-Brown1994-1} (see also \cite[Lemma 17]{Heath-Brown1994}) and Fouvry and Kl\"{u}ners  \cite[Section 4.2]{Fouvry2006} have methods which will show that for 
$\E(\left(p^{p\textrm{-rank}(S)}\right)^k)$ on the order of $p^{k(k-1)/2}$, the moments do indeed determine the distribution of $p\textrm{-rank}(S)$.  (See also \cite[Lemma 8.1]{EVW09} for the case when all the Sur moments are $1$.)
However, this will at best determine the distribution of the $p$-ranks of the sandpile groups.

Since we would like to determine more than just the $p$-ranks of the sandpile group, in this section we prove that we can recover the distribution from the moments we have found. 
Heath-Brown's method \cite[Lemma 17]{Heath-Brown1994-1} is to construct an infinite matrix that lower-triangularizes the infinite matrix that gives the relevant system of linear equations equations.  Once the system is lower-triangular, it certainly has a unique solution.  The difficulty is to construct a matrix with entries that are sufficiently small (and you can prove are sufficiently small) so that all the infinite sums involved converge.  
Heath-Brown constructs his matrix with entries from Taylor expansions of analytic functions in one variable.

 Our approach is to develop a multi-variable version of Heath-Brown's method.  However, the size of our moments are on the boundary of where this approach will work, and the functions we construct in the following lemma are carefully optimized.  In particular, it is critical that only terms with 
 $d_2+\cdots + d_m \leq b_1$ appear in the Taylor expansion of $H_{m,p,b}(z)$ below.  We now construct the analytic functions of several variables whose Taylor coefficients we will use to lower-triangularize our system of equations.

\begin{lemma} \label{L:Hacts}
Given a positive integer $m$, a prime $p$, and $b\in \Z^m$ with $b_1\geq b_2\geq \dots \geq b_m$,
we have an entire analytic function in the $m$ variables $z_1,\dots,z_m$
$$
H_{m,p,b}(z)=\sum_{\substack{d_1,\dots,d_m \geq 0\\d_2+\cdots + d_m \leq b_1 }} a_{d_1,\dots,d_m} z_1^{d_1} \cdots z_m^{d_m}
$$
and a constant $E$ such that
$$
a_{d_1,\dots,d_m}\leq E p^{-b_1d_1 - \frac{d_1(d_1+1)}{2}}.
$$
Further, 
if $f$ is a partition of length $\leq m$ and  $f > b$ (in the lexicographic ordering), then  $H_{m,p,b}(p^{f_1}, p^{f_1+f_2}, \dots, p^{f_1+\dots+f_m})=0$.  If $f=b$, then 
$H_{m,p,b}(p^{f_1}, p^{f_1+f_2}, \dots, p^{f_1+\dots+f_m})\ne0.$
\end{lemma}

\begin{proof}
We define analytic functions
$$
G(z_1):=\prod_{j\geq b_1+1}(1-\frac{z_1}{p^j}) =\sum_{d_1\geq 0} c_{d_1} z^{d_1}
$$
and
\begin{align*}
&H(z_2,\dots,z_m)\\
:=&
\prod_{j= b_1+b_2+1}^{2b_1}(1-\frac{z_2}{p^j}) \prod_{j= b_1+b_2+b_3+1}^{b_1+2b_2}(1-\frac{z_3}{p^j}) \cdots
\prod_{j= b_1\dots+b_m+1}^{b_1+\dots+b_{m-2}+2b_{m-1}}(1-\frac{z_m}{p^j})=
\sum_{d_2,\dots,d_m \geq 0} e_{d_1,\dots,d_m} z_2^{d_2} \cdots z_m^{d_m}.
\end{align*}
In each of the $z_i$ separately, for $2\leq i \leq m$, we have that $H$ is a polynomial of degree $b_{i-1}-b_i$.
We then have an entire, analytic function
in  $m$ variables
$$
H_{m,p,b}(z):=G(z_1)H(z_2,\dots,z_m)=\sum_{\substack{d_1,\dots,d_m \geq 0\\
d_2+\dots+d_m\leq b_1
}} a_{d_1,\dots,d_m} z_1^{d_1} \cdots z_m^{d_m}.
$$

We now estimate the size of the $a_d$.  
We see that $a_d=c_{d_1} e_{d_2,\dots d_m}$.
We have that 
$
G(pz)=(1-\frac{z}{p^{b_1}})G(z).
$
So
$
c_np^n=c_n-p^{-b_1}c_{n-1}.
$
Thus 
$
c_n= - \frac{p^{-b_1}c_{n-1}}{p^n-1}, 
$
and by induction,
$
c_n(-1)^n \frac{p^{-b_1n}}{\prod_{i=1}^n (p^i-1)}.
$
So
$
|c_n|\leq p^{-b_1n - \frac{n(n+1)}{2}}  \prod_{i\geq 1} (1-p^{-i})^{-1} .
$
Thus, 
$$
a_d\leq \frac{1}{\prod_{i\geq 1} (1-p^{-i})} p^{-b_1d_1 - \frac{d_1(d_1+1)}{2}} \max_{d_2,\dots,d_m}e_{d_2,\dots d_m}.
$$

Now we check the final statements of the lemma.
If $f>b$, suppose $f_i=b_i$ for $i\leq t$ and $f_{t+1} > b_{t+1}$ for some $0\leq t \leq m-1$.
Then, in particular $f_1+\dots+ f_i = b_1+\dots+ b_i$ for $i \leq t$, and 
$f_1+\dots+ f_{t+1} \geq b_1+\dots+ b_{t+1}+1$.
However, (when $t\geq 1$) since $f_{t+1}\leq f_t=b_t,$ we have $f_1+\dots+ f_{t+1} \leq  b_1+\dots+ b_{t-1}+2b_{t}.$
Since $H_b$ vanishes whenever $z_{t+1}=p^{k}$ for integers $k$ with $b_1+\dots+ b_{t+1}+1 \leq k
\leq b_1+\dots+ b_{t-1}+2b_{t},$ we obtain the desired vanishing.

For the last statement, we first note that since the  product in the definition of $G$ is absolutely
convergent, we have that $z_1=p^{b_1}$ is not a root of $G$.  Then we observe all the other finitely many factors in $H$ are non-zero in this case as well.

\end{proof}

In this theorem, we will use the analytic functions constructed in Lemma~\ref{L:Hacts} to construct an infinite matrix that will lower-triangularize the system of equations in \eqref{E:infsys} below.

\begin{theorem}\label{T:Momdet}
Let $p_1,\dots,p_s$ be distinct primes.  Let $m_1,\dots,m_s\geq 1 $ be integers.
Let $M_j$ be the set of partitions $\lambda$ at most $m_j$ parts.  
Let $M=\prod_{j=1}^{s} M_j$.  For $\mu\in M$, we write $\mu^j$ for its $j$th entry, which is a partition consisting of non-negative integers $\mu^j_i$ with $\mu^j_1\geq\mu^j_2\geq\dots \mu^j_{m_j}$.
 Suppose we have non-negative reals $x_\mu, y_\mu$, for each tuple of partitions $\mu\in M$.
Further suppose that we have non-negative reals $C_\lambda$ for each $\lambda\in M$ such that
$$C_\lambda\leq \prod_{j=1}^{s} {F^{m_j}p_j^{\sum_i \frac{\lambda^{j}_i(\lambda^{j}_i-1)}{2}}},$$
where $F>0$ is an absolute constant.
Suppose that 
for all $\lambda\in M$,
\begin{equation}\label{E:infsys}
\sum_{\mu\in M} x_\mu \prod_{j=1}^{s}  p_j^{\sum_i \lambda^{j}_i \mu^{j}_i} =\sum_{\mu\in M} y_\mu \prod_{j=1}^{s} p_j^{\sum_i \lambda^{j}_i \mu^{j}_i} =C_\lambda.
\end{equation}
Then for all $\mu$, we have that $x_\mu=y_\mu$.
\end{theorem}

\begin{proof}
 We will induct on the size of $\mu$ in the lexicographic total ordering (we take the lexicographic ordering for partitions and then the lexicographic ordering on top of that for tuples of partitions).  Suppose we have $x_\pi=y_\pi$ for every $\pi<\nu$.

We use Lemma~\ref{L:Hacts} to find $H_{m_j,p_j,\nu^j}(z)=\sum_d a(j)_d z_1^{d_1}\dots z_{m_j}^{d_{m_j}}.$
For $\lambda\in M$, we define
$$
A_\lambda:=\prod_{j=1}^{s} a(j)_{\lambda^j_1-\lambda^j_2,\lambda^j_2-\lambda^j_3,\dots, \lambda^j_{m_j}}.
$$ 
We wish to show that the sum
$
\sum_{\lambda\in M} A_\lambda C_\lambda
$
converges absolutely.
We have
\begin{align*}
 \sum_{\lambda\in M} |A_\lambda C_\lambda| 
&\leq\sum_{\lambda\in M} \prod_{j=1}^{s} \left| a(j)_{\lambda^j_1-\lambda^j_2,\lambda^j_2-\lambda^j_3,\dots, \lambda^j_{m_j}}{F^{m_j}p_j^{\sum_i \frac{\lambda^{j}_i(\lambda^{j}_i-1)}{2}}}  \right|\\
&= \prod_{j=1}^{s} \sum_{\lambda\in M_j}  \left|a(j)_{\lambda_1-\lambda_2,\lambda_2-\lambda_3,\dots, \lambda_{m_j}}{F^{m_j}p_j^{\sum_i \frac{\lambda_i(\lambda_i-1)}{2}}}  \right|.
\end{align*}

We now investigate the inner sum.  We drop the $j$ index, and let $b=\nu^j$.  We apply Lemma~\ref{L:Hacts} to obtain
\begin{align*}
&\sum_{\substack{d_1,\dots,d_m \geq 0\\
d_2+\dots+d_m \leq b_1
}} |a(j)_{d_1,d_2,\dots, d_m}| F^{m}p^{\sum_i \frac{\sum_{k=i}^{m}d_k (\sum_{k=i}^{m}d_k-1)}{2}}  
\leq \sum_{\substack{d_1,\dots,d_m \geq 0\\
d_2+\dots+d_m \leq b_1
}} E p^{-b_1 d_1 -\frac{d_1(d_1+1)}{2}} F^{m}p^{\sum_i \frac{\sum_{k=i}^{m}d_k (\sum_{k=i}^{m}d_k-1)}{2} }  .
%
\end{align*}
For each choice of $d_2,\dots d_m$, the remaining sum over $d_1$ is a constant times
$
\sum_{d_1\geq 0} p^{d_1(-b_1-1+d_2+\dots+d_m)},
$
which converges, so it follows that $\sum_{\lambda\in M} A_\lambda C_\lambda$ converges absolutely.

Suppose we have $x_\mu$ for $\mu \in M$ all non-negative, such that for all $\lambda\in M$,
$$
\sum_{\mu\in M} x_\mu \prod_{j=1}^{s}  p_j^{\sum_i \lambda^{j}_i \mu^{j}_i}=C_\lambda.
$$
So we have that
$$
\sum_{\lambda\in M} \sum_{\mu\in M} A_\lambda  x_\mu \prod_{j=1}^{s}  p_j^{\sum_i \lambda^{j}_i \mu^{j}_i}
$$
converges absolutely.
Thus,
\begin{align*}
 \sum_{\lambda\in M} A_\lambda C_\lambda &=\sum_{\lambda\in M} \sum_{\mu\in M} A_\lambda  x_\mu \prod_{j=1}^{s}  p_j^{\sum_i \lambda^{j}_i \mu^{j}_i}\\
&=\sum_{\mu\in M} x_\mu \sum_{\lambda\in M}  A_\lambda  \prod_{j=1}^{s}  p_j^{\sum_i \lambda^{j}_i \mu^{j}_i}\\
&=\sum_{\mu\in M} x_\mu  \prod_{j=1}^{s}  \sum_{\lambda\in M_j}  
a(j)_{\lambda_1-\lambda_2,\lambda_2-\lambda_3,\dots, \lambda_{m_j}}
  p_j^{\sum_i \lambda_i \mu^{j}_i}\\
\end{align*}
Now we consider the inner sum.  Again we drop the $j$ indices.  
We have
\begin{align*}
\sum_{\lambda\in M_j}  
a(j)_{\lambda_1-\lambda_2,\lambda_2-\lambda_3,\dots, \lambda_{m}}
  p^{\sum_i \lambda_i \mu_i}&=
  \sum_{d_1,\dots,d_{m}\geq 0}  a(j)_{d_1,\dots,d_m}  (p^{\mu_1})^{d_1}(p^{\mu_1+\mu_2})^{d_2} \cdots (p^{\mu_1+\dots+\mu_m})^{d_m}
  \\
&= H_{m,p,\nu}(p^{\mu_1}, p^{\mu_1+\mu_2}, \dots, p^{\mu_1+\dots+\mu_m}).
\end{align*}
If $\mu > \nu$ (in the lexicographic total ordering), then 
some $\mu^j>\nu^j$ and so for $m=m_j$ and $p=p_j$,
by Lemma~\ref{L:Hacts}, 
 $H_{m,p,\nu^j}(p^{\mu_1}, p^{\mu_1+\mu_2}, \dots, p^{\mu_1+\dots+\mu_m})=0$.  Further, if $\mu=\nu$, then for each (implicit) $j$ we have 
$H_{m,p,\nu}(p^{\mu_1}, p^{\mu_1+\mu_2}, \dots, p^{\mu_1+\dots+\mu_m})\ne0.$
So for some non-zero $u$,
$$
 \sum_{\lambda\in M} A_\lambda C_\lambda  =  x_\nu u 
+ \sum_{\mu\in M, \mu < \nu } x_\mu \sum_{\lambda\in M}  A_\lambda  \prod_{j=1}^{s}  p_j^{\sum_i \lambda^{j}_i \mu^{j}_i}.
$$
So since by assumption $x_\mu$ with $\mu<\nu$ we determined by the $C_\lambda$, we conclude that
$x_\nu$ is determined as well.

\end{proof}

In the following theorem we achieve two things.  We translate solving the above studied system of linear equations into finding the distribution of our random groups given their moments.  We also deal with the issue that we don't technically have moments of a distribution, but rather limits of moments of a sequence of distributions.  An important ingredient in solving this issue is showing that in our case, for any particular equation, we can use bounds coming from other equations to show that we satisfy the hypotheses of the Lebesgue Dominated Convergence Theorem.

\begin{theorem}\label{T:MomDetDetail}
Let $X_n$ be a sequence of random variables taking values in finitely generated abelian groups.
Let $a$ be a positive integer and $A$ be the set of (isomorphism classes of) abelian groups with exponent dividing $a$.
Suppose that for every $G\in A$, we have
$$
\lim_{n\ra \infty} \E(\# \Sur(X_n, G)) = |\wedge^2 G|. 
$$
Then for every $H\in A$,
 the limit
$
\lim_{n\ra\infty} \P(X_n\tensor \Z/a\Z \isom H)
$
exists, and for all $G\in A$ we have
$$
\sum_{H\in A} \lim_{n\ra\infty} \P(X_n\tensor \Z/a\Z \isom H) \#\Sur(H,G)=|\wedge^2 G|.
$$

 Suppose $Y_n$ is a sequence of random variables taking values in finitely generated abelian groups
 such that for every $G\in A$, we have
$$
\lim_{n\ra \infty} \E(\# \Sur(Y_n, G)) = |\wedge^2 G|. 
$$
Then, we have that for every every $H\in A$
$$
\lim_{n\ra\infty} \P(X_n\tensor \Z/a\Z \isom H) =\lim_{n\ra\infty} \P(Y_n\tensor \Z/a\Z \isom H).
$$
\end{theorem}
\begin{proof}
First, we will suppose that the limits
$
\lim_{n\ra\infty} \P(X_n\tensor \Z/a\Z \isom H)
$
exist, and from that show that
$$
\sum_{H\in A} \lim_{n\ra\infty} \P(X_n\tensor \Z/a\Z \isom H) \#\Sur(H,G)=|\wedge^2 G|.
$$

For each $G \in M$, we claim we can find an abelian group $G'\in M$ such that
$$
\sum_{H\in A} \frac{\#\Hom(H,G)}{\#\Hom(H,G')}
$$
converges.  We can factor over the primes $p$ dividing $a$, and reduce to the problem when $a=p^e$.  
Then if $G$ has type $\lambda$, we take $G'$ 
of type $\pi$ with $\pi_i'=2\lambda_i'+1$ for $1\leq i \leq e$.
Then we see
$$
\sum_{c_1\geq  \dots \geq c_e \geq 0} p^{\sum_{i=1}^e c_i(\lambda_i'-2\lambda_i'-1)}
=\sum_{c_1\geq  \dots \geq c_e \geq 0} p^{\sum_{i=1}^e c_i(-\lambda_i'-1)}
$$
converges.

There is some constant $D_G$ such that for all $n$ we have
$$
\P(X_n\tensor \Z/a\Z \isom H) \#\Hom(H,G') \leq \sum_{H\in A} \P(X_n\tensor \Z/a\Z \isom H) \#\Hom(H,G') \leq 
D_G.
$$
Thus, for all $n$,
$$
\P(X_n\tensor \Z/a\Z \isom H) \#\Hom(H,G)\leq 
D_G\#\Hom(H,G) \#\Hom(H,G')^{-1} .
$$
Since
$
\sum_{H\in A } D_G \#\Hom(H,G)\#\Hom(H,G')^{-1} 
$
converges, by the Lebesgue Dominated Convergence Theorem we have
$$
\sum_{H\in A } \lim_{n\ra\infty } \P(X_n\tensor \Z/a\Z \isom H) \#\Hom(H,G)
= \lim_{n\ra\infty } \sum_{H\in A } \P(X_n\tensor \Z/a\Z \isom H) \#\Hom(H,G).
$$
As this holds for every $G\in A$, we also have (by a finite number of additions and subtractions)
\begin{align*}
\sum_{H\in A } \lim_{n\ra\infty } \P(X_n\tensor \Z/a\Z \isom H) \#\Sur(H,G)
&= \lim_{n\ra\infty } \sum_{H\in A } \P(X_n\tensor \Z/a\Z \isom H) \#\Sur(H,G)\\
&= |\wedge^2 G|.
\end{align*}

Next, we show that if for every $G\in A$, 
\begin{align*}
\sum_{H\in A } \lim_{n\ra\infty } \P(X_n\tensor \Z/a\Z \isom H) \#\Sur(H,G)
&= \sum_{H\in A } \lim_{n\ra\infty } \P(Y_n\tensor \Z/a\Z \isom H) \#\Sur(H,G)\\
&= |\wedge^2 G|,
\end{align*}
then we have for every $H\in A$ that
$
\lim_{n\ra\infty } \P(X_n\tensor \Z/a\Z \isom H)=\lim_{n\ra\infty } \P(Y_n\tensor \Z/a\Z \isom H).
$
For each $G$, by a finite number of additions we have 
\begin{align*}
\sum_{H\in A } \lim_{n\ra\infty } \P(X_n\tensor \Z/a\Z \isom H) \#\Hom(H,G)
&= \sum_{H\in A } \lim_{n\ra\infty } \P(Y_n\tensor \Z/a\Z \isom H) \#\Hom(H,G)\\&=\sum_{G_1 \textrm{ subgroup of } G } |\wedge^2 G_1|.
\end{align*}
Now we will explain how to apply Theorem~\ref{T:Momdet} to conclude that 
$
\lim_{n\ra\infty } \P(X_n\tensor \Z/a\Z \isom H)=\lim_{n\ra\infty } \P(Y_n\tensor \Z/a\Z \isom H).
$
We factor $a=\prod_{j=1}^s p_j^{m_j}$.  The partition $\lambda^j\in M_j$ is the transpose of the type of the Sylow $p_j$-subgroup of $H$, which gives a bijection between $M$ and $A$.
We have that for $G\in A$ with corresponding $\lambda\in M$,
$$
C_\lambda=\sum_{G_1 \textrm{ subgroup of } G } |\wedge^2 G_1|\leq \prod_{j=1}^s {F^{m_j}p_j^{\sum_i \frac{\lambda^{j}_i(\lambda^{j}_i-1)}{2}}}.
$$
by Lemma~\ref{L:BoundMom}.  
For $H,G\in A$ with corresponding $\mu,\lambda\in M$, we have
$
\# \Hom(H,G) = \prod_{j=1}^s p_j^{\sum_i \lambda_i^j \mu_i^j }.
$
So for $H\in A$ with corresponding $\mu\in M$, we let
$
x_\mu:= \lim_{n\ra\infty } \P(X_n\tensor \Z/a\Z \isom H)
$
and similarly for $y_\mu$ and we can apply Theorem~\ref{T:Momdet}.

Now, we suppose for the sake of contradiction that the limit
$
\lim_{n\ra\infty} \P(X_n\tensor \Z/a\Z \isom H)
$
does not exist for at least some $H\in M$.  Then we can use a diagonal argument to find a subsequence
of $X_n$ where the limits do exist for all $H\in M$, and then another subsequence where the limits
do also exist for all $H \in M$, but at least one is different.
But since in each subsequence the limits
$
\lim_{n\ra\infty} \P(X_{i_n}\tensor \Z/a\Z \isom H)
$
exist, so we can use the above the conclude that these limits have to be the same for both subsequences, a contradiction.

\end{proof}

\section{Comparison to uniform random matrices}\label{S:Haar}

Above we have seen that the moments we have determined for sandpile groups of random graphs in particular imply many well-defined asymptotic statistics of the sandpile groups, but we have not yet determined the values of these statistics.  From Theorem~\ref{T:MomMat} we see these same moments hold for cokernels of a wide class of random matrices, in particular uniform random matrices over $\Z/a\Z$.
(Though the moments for cokernels of uniform random matrices follow from Theorem~\ref{T:MomMat},  there is a much simpler proof for the uniform case given in \cite{Clancy2014}.)
%
We can then use computations in the  uniform case to give us our desired statistics.
\begin{corollary}[of Theorems~\ref{T:Mom}, \ref{T:MomMat}, and \ref{T:MomDetDetail}] \label{C:first}
Let $G$ be a finite abelian group of exponent dividing $a$. Let $\Gamma\in G(n,q)$ be a random graph with sandpile group $S$. Let $H_n$ be a uniform random $n \times n$ symmetric matrix with entries in $\Z/a\Z$.
$$
\lim_{n\ra\infty} \P(S\tensor \Z/a\Z \isom G) =\lim_{n\ra\infty} \P(\cok(H_n) \isom G).
$$
\end{corollary}

In particular, we can conclude the following, which proves Theorem~\ref{T:Main}.
\begin{corollary}\label{C:Main}
Let $G$ be a finite abelian group. Let $\Gamma\in G(n,q)$ be a random graph with sandpile group $S$.
Let $P$ be a finite set of primes including all those dividing $|G|$.
Let $H_n$ be a random $n \times n$ symmetric matrix with entries in $\prod_{p\in P}\Z_p$ with respect to Haar measure.
Let $S_P$ be the sum of the Sylow $p$-subgroups of $S$ for $p\in P$.
  Then
\begin{align*}
\lim_{n\ra\infty} \P(S_P\isom G) &=\lim_{n\ra\infty} \P(\cok(H_n) \isom G)\\
&= \frac{\#\{\textrm{symmetric, bilinear, perfect $G\times G\ra \C^*$} \}}{|G||\Aut(G)|}
\prod_{p\in P}
\prod_{k\geq 0}(1-p^{-2k-1}).
\end{align*}
\end{corollary}
\begin{proof}
Note that if $G$ is a finite abelian group with exponent that has prime factorization $\prod_{p\in P} p^{e_p}$, then if we take $a=\prod_{p\in P} p^{e_p+1}$, for any finitely generated abelian group $H$,
with $H_P$  the sum of the Sylow $p$-subgroups of $H$ for $p\in P$, we have
$$
H\tensor \Z/a\Z \isom G\textrm{ if and only if } H_P \isom G.
$$
So the first equality follows from Corollary~\ref{C:first}.  

For the second equality, note that everything factors over $p\in P$, and so we can reduce to the case when $G$ is a $p$-group.
Let $\Phi_G$ be the set of symmetric, bilinear, perfect pairings $G\times G\ra \C^*$. 
For $\phi\in\Phi_G$, we let
$\Aut(G,\phi)$ be the set of automorphisms of $G$ that respect the pairing $\phi$.
Then $\Aut(G)$ acts naturally on $\Phi_G$, with orbits the isomorphism classes of 
symmetric, bilinear, perfect pairings $G\times G\ra \C^*$, and stabilizers 
$\Aut(G,\phi)$ for $\phi$ in the isomorphism class.  Let $\bar{\Phi}_G$ be the set of isomorphism classes of 
symmetric, bilinear, perfect pairings $G\times G\ra \C^*$.

Then \cite[Theorem 2]{Clancy2014} shows that
$$
\lim_{n\ra\infty} \P(\cok(H_n) \isom G)=\sum_{[\phi]\in \bar{\Phi}_G }\frac{1}{|G||\Aut(G,\phi)|} \prod_{k\geq 0}(1-p^{-2k-1}).
$$
By the orbit-stabilizer theorem, we have
$$
\sum_{[\phi]\in \bar{\Phi}_G}\frac{1}{\#\Aut(G,\phi)} =\sum_{\phi\in \Phi_G} \frac{1}{|\Aut(G)|}.
$$  We conclude the second equality of the corollary.
\end{proof}

In particular, this lets us see that any particular group appears asymptotically with probability $0$.

\begin{corollary}\label{C:Zero}
Let $G$ be a finite abelian group.  Let $\Gamma\in G(n,q)$ be a random graph with sandpile group $S$.
Then
$$
\lim_{n\ra \infty} \P(S\isom G)=0.
$$
\end{corollary}

\begin{proof}
Let $P_N$ be the set of primes $\leq N$ \emph{not} dividing the order of $G$.
Let $S_N$ be the sum of the Sylow $p$-subgroups of $S$ for $p\in P_N$.
Then
$$
\lim_{n\ra \infty}\P(S\isom G)\leq \lim_{n\ra \infty}\P(S_N \textrm{ trival})=\prod_{p\in P_N}
\prod_{k\geq 0}(1-p^{-2k-1}),
$$
where the last equality is by Corollary~\ref{C:Main}.
In particular, since the product
$
\prod_{p\in P_N} (1-p^{-1})
$
goes to $0$ as $N\ra \infty$, we can conclude the corollary.
\end{proof}

Also taking $a=p$ for a prime $p$ in Corollary~\ref{C:first}, we conclude the following on the distribution of $p$-ranks of sandpile groups.  
\begin{corollary}\label{C:prank}
Let $p$ be a prime.
 Let $\Gamma\in G(n,q)$ be a random graph with sandpile group $S$. Let $H_n$ be a uniform random $n \times n$ symmetric matrix with entries in $\Z/p\Z$.  Then for every non-negative integer $r$
$$
\lim_{n\ra\infty} \P(rank(S\tensor \Z/p\Z)=r) =\lim_{n\ra\infty} \P(rank(H_n)=n-r)=p^{-\frac{r(r+1)}{2}} \prod_{i=r+1}^{\infty} (1-p^{-i}) \prod_{i=1}^{\infty} (1-p^{-2i})^{-1} .
$$
\end{corollary}
\begin{proof}
We have the second equality because the number of symmetric $n\times n$ matrices over $\Z/p\Z$ with
rank $n-r$ is  (by \cite[Theorem 2]{MacWilliams1969})
$
p^{\frac{n(n+1)}{2}-\frac{r(r+1)}{2}} \prod_{i=1}^{\lfloor (n-r)/2 \rfloor} (1-p^{-2i})^{-1} \prod_{i=r+1}^{n} (1-p^{-i})
.
$

\end{proof}

We can also conclude an asymptotic upper bound on the probability that the sandpile group is cyclic.
\begin{corollary}\label{C:cyclic}
  Let $\Gamma\in G(n,q)$ be a random graph with sandpile group $S$.
Then
$$
\lim_{n\ra \infty} \P(S \textrm{ cyclic})\leq \zeta(3)^{-1}\zeta(5)^{-1}\zeta(7)^{-1}\zeta(9)^{-1}\cdots.
$$
\end{corollary}

\begin{proof}
Let $P_N$ be the set of primes $\leq N$.
Let $S_N$ be the sum of the Sylow $p$-subgroups of $S$ for $p\in P_N$.
Then
$$
\lim_{n\ra \infty}\P(S \textrm{ cyclic})\leq \lim_{n\ra \infty}\P(S_N \textrm{ cyclic}).
$$

We apply Corollary~\ref{C:first} with $a$ the product of the primes in $P_N$ and add over all $G$ cyclic with exponent dividing $a$.  We have
$$
\sum_{G \textrm{ cylic}, aG=0} \P(cok(H_n)\isom G)=\prod_{p\in P_N} \left( 
\P(cok(H_n \pmod{p} )\isom 1) +\P(cok(H_n \pmod{p} )\isom \Z/p\Z)
\right).
$$
By \cite[Theorem 2]{MacWilliams1969}, as above, we have
\begin{align*}
&\lim_{n\ra\infty} \left(\P(cok(H_n \pmod{p} )\isom 1) +\P(cok(H_n \pmod{p} )\isom \Z/p\Z) \right)\\=&
 \prod_{i=1}^{\infty} (1-p^{-2i})^{-1} \prod_{i=1}^{\infty} (1-p^{-i})
+
p^{-1} \prod_{i=1}^{\infty} (1-p^{-2i})^{-1} \prod_{i=2}^{\infty} (1-p^{-i})
\\
 =&
 \prod_{i=1}^{\infty}  (1-p^{-2i-1}) .
\end{align*}

So, 
$$
\lim_{n\ra \infty}\P(S \textrm{ cyclic})\leq \prod_{p\in P_N}   \prod_{i=1}^{\infty}  (1-p^{-2i-1}).
$$
Taking the limit as $N\ra\infty$, we obtain the corollary.
\end{proof}

Similarly, we can obtain an asymptotic upper bound for the probability that the number of spanning trees is square-free.
\begin{corollary}\label{C:sqf}
  Let $\Gamma\in G(n,q)$ be a random graph with sandpile group $S$.
Then
$$
\lim_{n\ra \infty} \P(|S| \textrm{ square-free})\leq \zeta(2)^{-1}\zeta(3)^{-1}\zeta(5)^{-1}\zeta(7)^{-1}\zeta(9)^{-1}\cdots.
$$
\end{corollary}

\begin{proof}
Let $P_N$ be the set of primes $\leq N$.
Let $S_N$ be the sum of the Sylow $p$-subgroups of $S$ for $p\in P_N$.
Then
$$
\lim_{n\ra \infty}\P(S \textrm{ square-free})\leq \lim_{n\ra \infty}\P(|S_N| \textrm{ square-free}).
$$
By summing Corollary~\ref{C:Main} over all $G$ such that $|G|$ has all prime factors in $P_N$ and $|G|$ is square-free, we have 
$$
\lim_{n\ra \infty}\P(|S_N| \textrm{ square-free}) =\prod_{p\in P_N} \left(  (1+p^{-1})\prod_{k\geq 0}(1-p^{-2k-1} )\right)
=\prod_{p\in P_N} \left(  (1-p^{-2})\prod_{k\geq 1}(1-p^{-2k-1} )\right).
$$
Taking the limit as $N\ra\infty$, we obtain the corollary.
\end{proof}

\begin{remark}\label{R:Mat}
Of course, all of the corollaries in this section also follow if $S$ is replaced by the cokernel of a random matrix satisfying the hypotheses of Theorem~\ref{T:MomMat} (using Theorem~\ref{T:MomMat} in place of Theorem~\ref{T:Mom}).  In this case, the analog of Corollary~\ref{C:prank} is a result announced by Maples \cite{Map13b}.
\end{remark}

It would be nice to know the rest of the limits for uniform random matrices that occur in Lemma~\ref{C:first}.  More specifically, let $H_n$ be a uniform random $n \times n$ symmetric matrix with entries in $\Z/a\Z$.  What is
$$
\lim_{n\ra\infty} \P(\cok(H_n) \isom G)?
$$
Above we have seen the answer when $a$ is a prime, and when every prime dividing the exponent of $G$ divides $a$ to at least one higher power than it divides the exponent of $G$.

\subsection*{Acknowledgements} 
The author would like to thank Sam Payne, Betsy Stovall,  Jordan Ellenberg, Philip Matchett Wood, Benedek Valko, and Steven Sam for useful conversations regarding the work in this paper, and Sam Payne, Philip Matchett Wood, Lionel Levine, Van Vu, Dino Lorenzini, and Karola M\'{e}sz\'{a}ros for helpful comments on the exposition.
This work was done with the support of an American Institute of Mathematics Five-Year Fellowship and National Science Foundation grants DMS-1147782 and DMS-1301690.

\newcommand{\etalchar}[1]{$^{#1}$}
\def\cprime{$'$}

\end{document}